\pdfoutput=1
\newcommand{\arxiv}[1]{#1}
\newif\ifMOR
\MORfalse
\newcommand{\JfAcc}{\mu}
\newcommand{\DirectionSet}[1]{\mathcal{D}_{r}(#1)}
\newcommand{\ConSet}[0]{[ \NumCon ]}

\newif\ifRepeatThms
\RepeatThmsfalse
\arxiv{
\documentclass[11pt]{article} 
}

\ifMOR
\documentclass[moor]{informs4} 

\fi

\arxiv{
\usepackage[hidelinks]{hyperref}
}
\usepackage{blindtext}
\usepackage[utf8]{inputenc}

\usepackage{graphicx,fancyhdr,amsmath,amssymb,subfig,url}

\arxiv{
\usepackage{amsthm}
}
\usepackage{algorithm}
\usepackage{algpseudocode}
\usepackage{booktabs}
\usepackage{float}
\restylefloat{table}
\usepackage[margin=1in]{geometry}
\usepackage{multirow}
\usepackage{enumitem}
\usepackage[normalem]{ulem}
\usepackage{thm-restate}
\usepackage{physics}
\usepackage{scrextend}
\usepackage[numbers,square]{natbib}
\usepackage{mathtools}

\usepackage[dvipsnames]{xcolor}
\newcommand{\CHECKED}{}
\newcommand{\edit}[1]{#1}
\iffalse
\newcommand{\editTwo}[1]{#1}
\newcommand{\editThree}[1]{{\color{BrickRed} #1}}
\newcommand{\editFour}[1]{{\color{green} #1}}
\else
\newcommand{\editTwo}[1]{#1}
\newcommand{\editThree}[1]{#1}
\newcommand{\editFour}[1]{#1}
\fi

\input{macros.sty} % input and not usapackage is intentional - TeXStudio 
\ifMOR
	\usepackage{natbib}
	\NatBibNumeric
	\bibpunct[, ]{[}{]}{,}{n}{}{,}%
	
	\usepackage[colorlinks=true,breaklinks=true,bookmarks=true,urlcolor=blue,
	citecolor=blue,linkcolor=blue,bookmarksopen=false,draft=false]{hyperref}
	
	\def\EMAIL#1{\href{mailto:#1}{#1}}% When hyperref is used, otherwise outcomment 
	\def\URL#1{\href{#1}{#1}}         % When hyperref is used, otherwise outcomment 
	
	\TheoremsNumberedThrough     % Preferred (Theorem 1, Lemma 1, Theorem 2)
	\EquationsNumberedThrough    % Default: (1), (2), ...
\fi

\newcommand{\paperTitle}{Worst-case iteration bounds for log barrier methods on problems with nonconvex constraints}

\newcommand{\yinyuContact}[0]{Management Science and Engineering, Stanford University (\href{mailto:yye@stanford.edu}{yyye@stanford.edu})}

\newcommand{\oliverContact}[0]{Management Science and Engineering,  Stanford University (\href{mailto:ohinder@pitt.edu}{ohinder@pitt.edu})}

\newcommand{\AbstractText}[0]{
Interior point methods (IPMs) that handle nonconvex constraints such as IPOPT, KNITRO and LOQO have had enormous practical success. We consider IPMs in the setting where the objective and constraints are thrice differentiable, and have Lipschitz first and second derivatives on the feasible region. We provide an IPM that, starting from a strictly feasible point, finds a $\mu$-approximate Fritz John point by solving $\mathcal{O}( \mu^{-7/4})$ trust-region subproblems. For IPMs that handle nonlinear constraints, this result represents the first iteration bound with a polynomial dependence on $1/\mu$. We also show how to use our method to find scaled-KKT points starting from an infeasible solution and improve on existing complexity bounds.
}

\begin{document}
\ifMOR
\TITLE{\paperTitle}
\ARTICLEAUTHORS{%
	\AUTHOR{Oliver Hinder\thanks{A significant portion of this work was done at Stanford where the first author was supported by the PACCAR Inc Stanford Graduate Fellowship and the Dantzig-Lieberman fellowship.}}
	\AFF{Industrial Engineering, University of Pittsburgh, \EMAIL{ohinder@pitt.edu}, \URL{}}
	\AUTHOR{Yinyu Ye}
	\AFF{Department of Management Science and Engineering, Stanford University, \EMAIL{yyye@stanford.edu}, \URL{}}
} % end of the block

\ABSTRACT{\AbstractText}

\MSCCLASS{90C30, 90C60}
\ORMSCLASS{Primary: nonlinear programming; secondary: computational complexity.}
\HISTORY{}
\fi

\arxiv{
	\title{\paperTitle}
	\author{Oliver Hinder\thanks{\oliverContact} \and Yinyu Ye\thanks{\yinyuContact}}
	\date{}
}

\def\workingPaper{0}

\maketitle

\ifMOR\else
\begin{abstract}
\AbstractText
\end{abstract}
\fi

\section{Introduction}\label{sec:intro}

This paper studies constrained optimization problems of the form:
\begin{flalign*}
\minimize_{x \in \reals^{\NumVar} }{~f(x)} \quad \text{such that} \quad \cons(x) \ge \zeros
\end{flalign*}
where $\reals$ is the set of real numbers, $\NumVar$ and $\NumCon$ are positive integers, $f : \reals^{\NumVar} \rightarrow \reals$ and $a : \reals^{\NumVar} \rightarrow \reals^{\NumCon}$ are thrice differentiable on $\reals^{n}$.

Providing worst-case bounds for solving this problem to global optimality is intractable even in the unconstrained case \cite{nemirovskii1983problem}. So instead we seek a notion of approximate local optimality.
The condition we primarily focus on is a
 Fritz John point \cite{john1948extremum}, a necessary condition for local optimality. This is defined as a point $(x,\lambda,t) \in \reals^{\NumVar} \times \reals^{\NumCon} \times \reals$ satisfying
\begin{subequations}
\label{eq:exact-FJ}
\begin{flalign}\label{eq:FJ}
(t, \cons(x), \lambda) &\ge  \zeros \\
\lambda_i \cons_i(x) &= 0  \quad \forall i \in \ConSet  \\
t \grad f(x) - \grad \cons(x)^T \lambda &= \zeros \\
(\lambda,t) &\neq \zeros,
\end{flalign}
\end{subequations}
where $\ConSet \defeq \{ 1, \dots, \NumCon \}$, $\lambda$ is a vector of dual variables, and $t$ is a scalar that is equal to one in the KKT conditions. When the Mangasarian-Fromovitz constraint qualification \cite{mangasarian1967fritz} holds, all Fritz John points are KKT points after appropriate scaling of the multipliers. Since it is not reasonable to expect a derivative-based iterative algorithm to find an exact Fritz John point, we require a notion of an approximate Fritz John point. The definition of an approximate Fritz John point we will use is
\begin{subequations}\label{eq:True-Fritz-John}
\begin{flalign}
(t, \cons(x), \lambda) &\ge  \zeros  
\label{eq:t-a-lambda-nonnegative}\\
\lambda_i \cons_i(x) &\le 2 \JfAcc \quad \forall i \in \ConSet \label{eq:True-Fritz-John:dual-gap} \\
\norm{ t \grad f(x) - \grad \cons(x)^T \lambda }_{2} &\le \JfAcc, 
\label{eq:True-Fritz-John:dual-feas} \\
\editTwo{t + \| \lambda \|_1} &\editTwo{= 1}
\label{eq:t-and-lambda-bounded}
\end{flalign}
\end{subequations}
where $\JfAcc \ge 0$ is a parameter measuring the accuracy of our approximation and a small $\JfAcc$ is desirable.
\editTwo{
Note that if we solve \eqref{eq:True-Fritz-John} with $\mu = 0$ then \eqref{eq:exact-FJ} is satisfied. 
Furthermore, if we consider any sequence $(t\seq{k}, x\seq{k}, \lambda\seq{k}, \mu\seq{k})$ \editFour{for $k \in \N$, where $\N$ is the set of natural numbers starting at one,} satisfying \eqref{eq:True-Fritz-John} with $\lim_{k \rightarrow \editFour{\infty}} \mu\seq{k} \rightarrow 0$, i.e., by annealing $\mu$. Then as the following Lemma shows, if $x\seq{k}$ converges (or any subsequence of $x\seq{k}$ converges)
to a point satisfying the Mangasarian-Fromovitz constraint qualification \cite{mangasarian1967fritz} then $x\seq{k}$ converges to a KKT point. Even stronger, there will be a corresponding convergent subsequence \editFour{$(t\seq{\pi(k)}, x\seq{\pi(k)}, \lambda\seq{\pi(k)}, \mu\seq{\pi(k)})$, where $\pi : \N \rightarrow \N$ is a strictly increasing function}, with limit $(t^{\star}, x^{\star}, \lambda^{\star}, 0)$, $t^{\star} > 0$ and dual multipliers generated by $\lambda^{\star} / t^{\star}$.

\begin{lemma}\label{lem:MFCQ}
Consider a sequence $(t\seq{k}, x\seq{k}, \lambda\seq{k}, \mu\seq{k})$ satisfying \eqref{eq:True-Fritz-John} with $\lim_{k \rightarrow \editFour{\infty}} \mu\seq{k} = 0$ and assume there exists some $x^{\star}$ such that $\lim_{k \rightarrow \infty} x\seq{k} = x^{\star}$.
Also, assume that at $x^{\star}$ the Mangasarian-Fromovitz constraint qualification \cite{mangasarian1967fritz} holds, i.e, there exists $v \in \reals^{\NumVar}$ such that $\grad a_i(x^{\star}) \cdot v > 0$ for all $i \in \mathcal{A} := \{ i \in [\NumCon] : a_i(x^{\star}) = 0 \}$. Then there exists a convergent subsequence $(t\seq{\pi(k)}, x\seq{\pi(k)}, \lambda\seq{\pi(k)}, \mu\seq{\pi(k)})$ such that $\lim_{k \rightarrow \editFour{\infty}} t\seq{\pi(k)} = \liminf_{k \rightarrow \infty} t\seq{k} > 0$.
\end{lemma}

\begin{myproof}
By definition of $\liminf$ there exists some subsequence $(t\seq{\pi'(k)}, x\seq{\pi'(k)}, \lambda\seq{\pi'(k)}, \mu\seq{\pi'(k)})$ such that $t^{\star} := \lim_{i \rightarrow \infty} t\seq{\pi'(i)} = \liminf_{k \rightarrow \infty} t\seq{k}$.
By \eqref{eq:t-a-lambda-nonnegative} and \eqref{eq:t-and-lambda-bounded}, $(t\seq{k}, \lambda\seq{k})$ is bounded, and both $x\seq{k}$ and $\mu\seq{k}$ are bounded by the assumption they have limits. Therefore, by the Bolzano-Weierstrass Theorem there must also exist a subsequence $(t\seq{\pi(k)}, x\seq{\pi(k)}, \lambda\seq{\pi(k)}, \mu\seq{\pi(k)})$ of the subsequence $(t\seq{\pi'(k)}, x\seq{\pi'(k)}, \lambda\seq{\pi'(k)}, \mu\seq{\pi'(k)})$ with a limit $(t^{\star}, x^{\star}, \lambda^{\star}, 0)$.
By \eqref{eq:t-a-lambda-nonnegative} and \eqref{eq:True-Fritz-John:dual-gap} we get $\lambda_i^\star = 0$ for all $i \not\in \mathcal{A}$.
To obtain a contradiction assume $t^\star = 0$.
By $t^\star = 0$, $\lambda_i^\star = 0$ for all $i \not\in \mathcal{A}$, \eqref{eq:t-a-lambda-nonnegative} and \eqref{eq:t-and-lambda-bounded} we get $\lambda^{\star} \ge \zeros$ and $\lambda_j^\star > 0$ for some $j \in \mathcal{A}$.
Let $v$ be the vector defined in the premise of the Lemma.
Then 
\begin{flalign*} 
0 = t^{\star} \grad f(x^{\star}) \cdot v =_{(a)} (\grad a(x^{\star})^T \lambda^{\star}) \cdot v =_{(b)} \sum_{i \in \mathcal{A}} \lambda_i^\star \grad a_i(x^{\star}) \cdot v \ge_{(c)} \lambda_j^\star \grad a_j(x^{\star}) \cdot v >_{(d)} 0
\end{flalign*} 
where $(a)$
uses \eqref{eq:True-Fritz-John:dual-feas} and the assumed differentiablity of $\obj$ and $\cons$, $(b)$ uses $\lambda_i^\star = 0$ for all $i \not\in \mathcal{A}$, $(c)$ uses that for all $i \in \mathcal{A}$ both $\grad a_i(x^{\star}) \cdot v > 0$ and $\lambda^{\star}_i \ge 0$ hold, and $(d)$ uses that $\grad a_j(x^{\star}) \cdot v > 0$ and $\lambda^\star_j > 0$. 
This gives a contradiction, thus $t^{\star} > 0$ as desired.
%However, $t^{\star} = 0$ implies $t^{\star} \grad f(x^{\star}) \cdot v = 0 \not> 0$ which is a contradiction, thus $t^{\star} > 0$ as desired.
\end{myproof}
}

Our approach is loosely inspired by feasible start interior point methods (IPMs) \cite{kojima1989primal,megiddo1989pathways,monteiro1989interior,renegar1988polynomial} and trust-region algorithms \cite{conn2000trust,sorensen1982newton}. 
To guide our trust-region method we use the log barrier,
\begin{flalign}\label{barrier-problem}
\barrier(x) := f(x) - \mu \sum_{i=1}^{\NumCon} \log(\cons_i(x))
\end{flalign}
with some parameter $\mu > 0$, and start from a strictly feasible point, i.e., 
\[
x\ind{0} \in \X \defeq \{ x \in \reals^{n} : \cons(x) > \zeros \}.
\]
The log barrier penalizes points too close to the boundary, enabling the use of unconstrained methods to solve a constrained problem. Typically, if $f$ and each $\cons_i$ were linear we would apply Newton's method to the log barrier. However, since we allow $a_i$ to be nonlinear, $\grad^2 \barrier$ could be singular or indefinite. To circumvent this issue, we employ a trust-region method to generate our search directions:
$$
\dir{x} \in \argmin_{u  \in \ball{r}{\zeros}}{ \model_x^{\barrier}(u)  }
$$
with
\begin{flalign*}
\model_x^{\barrier}(u) &:= \frac{1}{2} u^T \grad^2 \barrier (x) u + \grad \barrier(x)^T u \\
\ball{r}{v} &:= \{ x \in \reals^{\NumVar} : \| x - v \|_2 \le r \}.
\end{flalign*}
The function $\barrierModel_x ( u )$ is a second-order Taylor series local approximation to \editFour{$\barrier(x + u) - \barrier(x)$} at $x$. %It predicts how much $\barrier$ changes as we move from $x$ to $x + u$. 

%One might ask why we are interested

\paragraph{Outline} The remainder of the introduction provides notation and reviews related work.
 Section~\ref{sec:our-trust-region-ipm} introduces our main algorithm, a trust-region IPM. Section~\ref{sec:elementary} gives a series of useful lemmas for the analysis. Section~\ref{sec:worst-case} proves our main result. Section~\ref{sec:comparison-with-existing-results} shows how to remove the assumption that we are given a \editThree{strictly feasible starting point} and compares the iteration bounds of our IPM with existing iteration bounds for problems with nonconvex constraints \cite{birgin2016evaluation,cartis2020strong,cartis2011evaluation,cartis2014complexity}.

%

%\section{Preliminaries}\label{sec:basic-facts}
\subsection{Preliminaries}\label{sec:defs}

\paragraph{Notation}
Let $\diag(v)$ be a diagonal matrix with entries composed of the vector $v$. Let $\reals$ denote the set of real numbers, $\Rp$ the set of nonnegative real numbers and $\Rpp$ the set of strictly positive real numbers. Let $\Convex\{ x, y\} = \{ \alpha x + (1 - \alpha) y : \alpha \in [0,1] \}$. \editThree{For a matrix $M$} let $\lambda_{\min}(M)$ denote the minimum eigenvalue of a matrix \editThree{and $\| M \|_2$ the spectral norm}. Unless otherwise specified, $\log(\cdot)$ is the natural logarithm. 
Define the Lagrangian as $\Lag(x,y) := f(x) - y^T \cons(x)$.
For a $p$th order differentiable function $g : \reals \rightarrow \reals$, we let $g^{(p)}(\theta)$ denote $\frac{\partial^{p} g(\theta)}{\partial \theta^{p}}$.

\begin{definition}\label{def:lip}(Lipschitz derivatives)
	Let $\LipP \in (0,\infty)$ be a constant and $p$ a nonnegative integer.	
	A univariate function $g : \reals \rightarrow \reals$ has
	$\LipP$-Lipschitz $p$th derivatives on the set $S \subseteq \reals$, if for all $[\theta^{1}, \theta^{2}] \subseteq S$ we have $\abs{g^{(p)}(\theta^{1}) - g^{(p)}(\theta^{2})} \le \LipP \abs{\theta^{1} - \theta^2}$.
	A multivariate function $w : \reals^{n} \rightarrow \reals$ has $\LipP$-Lipschitz $p$th derivatives on the set $S \subseteq \reals^{n}$ if for any $x \in S$ and $v \in \ball{1}{\zeros}$ the univariate function $g : \reals \rightarrow \reals$ defined by $g(\theta) := w(x + v \theta)$ has $\LipP$-Lipschitz $p$th derivatives on the set $\{ \theta \in \reals : x + v \theta \in S \}$.
\end{definition}

We often refer to the function $a : \reals^{\NumVar} \rightarrow \reals^{\NumCon}$ as having $\LipP$-Lipschitz $p^{th}$ derivatives on the set $S$. By this we mean that each component function $a_i$ has $\LipP$-Lipschitz $p^{th}$ derivatives on the set $S$. Finally, the matrix $\grad \cons(x)$ is the $\NumCon \times \NumVar$ Jacobian of $\cons(x)$ where the $i$th row consists of $\grad a_i(x)$.

Our main result in Section~\ref{sec:worst-case} is proven under the following assumptions.

\begin{assumption}(Lipschitz derivatives)\label{assume-lip-deriv}
	The functions $f$ and each $a_i$ for $i \in [m]$ is thrice differentiable on $\reals^{\NumVar}$.
	Let $\LipGrad, \LipHess \in (0,\infty)$.
	On the set $\X$, $f$ and each  $\cons_i$ have $\LipGrad$-Lipschitz first derivatives and $\LipHess$-Lipschitz second derivatives.
\end{assumption}
 
\begin{assumption}\label{assume:barrier-and-initial-point}
\editFour{For a given $\mu \in (0,\infty)$, i.e., the $\mu$ supplied to Algorithm~\ref{algIPM},}
the barrier function is bounded below: $\barrier^{*} \defeq \inf_{x \in \X} \barrier(x) > -\infty$. \editFour{Also,} a strictly feasible starting point is provided, i.e., $x\ind{0} \in \X$.
\end{assumption}

This assumption that $f$ and each $a_i$ is thrice differentiable on $\reals^{\NumVar}$ allows us to apply Lemma~\ref{lem:simple-lip}. 
In particular, we can do this because the set $\X$ is open as $a$ is continuous on $\reals^{n}$ (by thrice differentiability on $\reals^n$). Furthermore, without this additional assumption or something similar, the function $a$ could be discontinuous on the boundary of $\X$, which would break our proofs.

\begin{lemma}\label{lem:simple-lip}
	A univariate function $g : \reals \rightarrow \reals$ that is $p+1$ order differentiable on the open set $S  \subseteq \reals$ has $\LipP$-Lipschitz $p$th order derivatives on $S$ if and only if $\abs{g^{(p+1)}(\theta)} \le \LipP$ for all $\theta \in S$.
\end{lemma}

\begin{myproof}
	The `if' follows by
	$\abs{g^{(p)}(\theta^1) - g^{(p)}(\theta^{2}) } = \abs{ \int^{\theta^1}_{\theta_2} g^{(p+1)}(\theta) ~ d \theta }  \le \LipP \abs{\theta^{1} - \theta^{2}}$. 
	The `only if' uses that $S$ is open and therefore, $\abs{g^{(p+1)}(\theta)} = \abs{ \lim_{h \rightarrow 0}  \frac{g^{(p)}(\theta+h) - g^{(p)}(\theta)}{h}} \le \lim_{h \rightarrow 0} \abs{\frac{\LipP \abs{h}}{h}}  = \LipP$.
\end{myproof}

Key to our results is Taylor's theorem. 
Taylor's theorem states that given a $p+1$ differentiable one-dimensional function $g : \reals \rightarrow \reals$, if it's $p$th order derivatives are $\LipP$-Lipschitz on the interval $[0,\theta]$, then for all $q \in \{0, \dots, p \}$ one has
\begin{flalign}\label{eq:taylor}
\abs{ \sum_{i=0}^{p-q} \theta^{i}  \frac{g^{(q+i)}(0)}{i!} - g^{(q)}(\theta) } \le \frac{\LipP \abs{\theta}^{1+p-q}}{(1 + p - q)!}.
\end{flalign}
See \cite[Theorem~50.3]{calc} for a proof of the remainder version of this theorem with $q = 0$. To extend this theorem to $q > 0$ it suffices to apply the theorem to the function $h(\theta) \defeq g^{(q)}(\theta)$.

A well-known consequence of Lemma~\ref{lem:simple-lip} we will frequently use is that 
if $f : \R^{\NumVar} \rightarrow \R$ is twice differentiable and has $\LipP$-Lipschitz derivatives for $p \in \{0,1\}$ then $\| \grad_{p+1} f(x) \|_2 \le \LipP$ for all $x \in \R^{\NumVar}$.

\subsection{Related work and motivation}\label{sec:related-work}

The practical performance of IPMs is excellent for linear \cite{mehrotra1992implementation}, conic \cite{sturm2002implementation}, general convex \cite{andersen1998computational}, and nonconvex optimization \cite{byrd2006knitro, vanderbei1999loqo, wachter2006implementation}. Moreover, the theoretical performance of IPMs for linear \cite{karmarkar1984new,renegar1988polynomial,ye1991n,ye1994nl,zhang1994convergence} and conic \cite{nesterov1994interior} optimization is well-studied. The main theoretical result in this area is that it takes at most $\mathcal{O}( \sqrt{c} \log(1/ \epsilon  ) )$ iterations to find an $\epsilon$-global minimum \cite{nesterov1994interior}, where $c$ is the self-concordance parameter (e.g., $c = m + n$ for linear programming). Each IPM iteration consists of a Newton step applied to an unconstrained \editThree{or linearly constrained} optimization problem. Unfortunately, this approach  only works for convex cones with self-concordant barriers. 

While self-concordance theory is designed for structured convex problems, there is a rich body of literature on the minimization of general unconstrained objectives, particularly if the objective is convex \cite{nemirovskii1983problem,nesterov1983method}. Here, we briefly review results in nonconvex optimization since it is most relevant to our work. In unconstrained nonconvex optimization, the measure of local optimality is usually whether $\| \nabla f(x) \|_2 \le \mu$, such a point $x$ is known as a $\mu$-approximate stationary point. 
A fundamental result is that gradient descent needs at most $\mathcal{O}(\mu^{-2})$ iterations to find a $\mu$-approximate stationary point on functions with Lipschitz continuous first derivatives.
\citet{nesterov2006cubic} show that cubic regularized Newton takes at most $\mathcal{O}(\mu^{-3/2})$ iterations to find a $\mu$-approximate stationary points on functions with Lipschitz continuous second derivatives. 
The same iteration bound can be extended to trust-region methods \cite{curtis2017trust,yeTrust}.  
These $\mathcal{O}(\mu^{-2})$ and $\mathcal{O}(\mu^{-3/2})$ iteration bounds match lower bounds for functions with Lipschitz continuous first and second derivatives respectively \cite{carmon2020lowerI,carmon2021lowerII}. 

There are few worst-case iteration bounds for nonconvex optimization with nonconvex constraints \cite{birgin2016evaluation,cartis2020strong,cartis2011evaluation,cartis2014complexity}. 
Moreover, despite the practical success of nonconvex IPM \cite{byrd2006knitro,vanderbei1999loqo,wachter2006implementation} there are no iteration bounds for these methods.
While there has been theoretical work studying IPMs handling nonlinear constraints, most of this work focuses on superlinear convergence in regions close to local optima \cite{ulbrich2004superlinear,vicente2002local} or tends to show only that the method eventually converges \cite{byrd2000trust, chen2006interior, conn2000primal, gould2015interior, hinder2018one, wachter2005line}  without giving explicit iteration bounds.

There do not exist iteration bounds for IPMs with general constraints, however,
there are results for the special case when the constraints are linear inequalities.
In particular, \cite{bian2015complexity,haeser2019optimality,ye1998complexity} consider an affine scaling technique for general objectives with linear inequality constraints, i.e., $\cons_i$ are linear. At each iteration they solve problems of the form
\begin{flalign}
\dir{x} \in \argmin_{u \in \reals^{\NumVar} : \| S^{-1} \grad a(x) u \|_2 \le r}{\model_x^{\barrier}(u)}
\end{flalign}
with $S = \diag(\cons(x))$. In this context, \citet*{haeser2019optimality} give an algorithm with an $\mathcal{O}(\mu^{-3/2})$ iteration bound for finding KKT points. This work is pertinent to ours, but the addition of nonconvex constraints and the use of a trust-region method instead of affine scaling distinguish our work. 

\section{Our trust-region IPM}\label{sec:our-trust-region-ipm}

\begin{algorithm}[t]
	\begin{algorithmic}[1]
		\Function{\AlgMain}{$\obj, \cons, \mu, \tau_{l}, \tau_{c}, \LipGrad, \eta, \x\ind{0}$}
		\State \textbf{Input:} $\grad \obj$ and $\grad \cons$ are $\LipGrad$-Lipschitz. A parameter $\eta \in (0,1)$. A starting point $\x\ind{0} \in \X$.
		\State $\x \gets \x\ind{0}$
		\For{$k = 0, \dots , \infty$}
		\State $(\xPlus,\yPlus) \gets \callStepIPM{f,a,\mu, x_k, \eta}$
		\If{$(\xPlus, \yPlus)$ satisfies \eqref{eq:first-order-SIP-full:first-order} and \eqref{eq:first-order-SIP-full:second-order}}
		\label{line:term}\State \Return{$(\xPlus, \yPlus)$} \Comment{Termination criteria met.}
		\Else
		\State $x \gets \xPlus$ \Comment{Only update primal variables, throw away new dual variable $y^{+}$.}
		\EndIf
		\EndFor
		\EndFunction
		\Function{\StepIPM}{$\obj, \cons, \mu, x, \eta$}\label{alg:ipm-step}
		\State $S \gets \diag(\cons(x))$
		\State $y \gets  \mu S^{-1} \ones$ \Comment{Primal update of dual variables.}
		\State $r \gets \frac{\eta}{2} \sqrt{\frac{\mu}{\LipGrad(1 + \| y \|_1)}}$ \label{eq:pick-r} \Comment{Trust-region radius gets smaller as the dual variables get larger.}
		\State $\dir{x} \in \argmin_{u \in \ball{r}{\zeros}}{ \barrierModel_{x} ( u ) }$
		\State $\dir{s} \gets  \grad \cons(x) \dir{x}$
		\State $\dir{y} \gets  - \mu S^{-2} \dir{s}$
		\State\label{line:alpha-choice} $\alpha \gets \min \left\{ \frac{\eta}{\| S^{-1} \dir{s} \|_2}, 1 \right\}$ \label{line:step-size-computation} \Comment{Pick a step size $\alpha \in (0,1]$ to guarantee $x^{+} \in \X$.}
		\State $x^{+} \gets x + \alpha \dir{x}$
		\State $y^{+} \gets y + \alpha \dir{y}$
		\State \Return{$(x^{+}, y^{+})$}
		\EndFunction
	\end{algorithmic}
	\caption{Adaptive trust-region interior point algorithm with fixed $\mu$}\label{algIPM}
\end{algorithm}

This section introduces our trust-region IPM (Algorithm~\ref{algIPM}).
A naive algorithm we could use is
\begin{flalign*}
\dir{x} &\in \argmin_{u \in \ball{r}{\zeros}} \barrierModel_{x} ( u ) \\
\xPlus &\gets x + \dir{x} \\
x &\gets \xPlus
\end{flalign*}
for some fixed constant $r \in (0,\infty)$ where $x$ denotes the current iterate and $x^{+}$ the next iterate. If $\grad^2 \barrier$ is $\LipHess$-Lipschitz then one can show convergence to an $\epsilon$-approximate stationary point of $\barrier$ in $\mathcal{O}(\LipHess^{1/2} \epsilon^{-3/2})$ iterations \cite{nesterov2006cubic}. Unfortunately, the log barrier is not Lipschitz continuous on the set of strictly feasible solutions so we must use different analysis techniques.
Instead, as per line~\ref{eq:pick-r}  of Algorithm~\ref{algIPM}, we make the trust-region radius adaptive to the size of the dual variables using the formula
$$
r \gets \frac{\eta}{2} \sqrt{\frac{\mu}{\LipGrad(1 + \| y \|_1)}},
$$
where the parameters $\eta \in (0,\infty)$ is a problem dependent parameter, we defer its choice to Theorem~\ref{thmMainResultNonconvex}.
This choice ensures that for constant $\eta \in (0,\infty)$ the trust-region radius becomes smaller as the dual variable size increases. This enables the algorithms to adapt to the `local' Lipschitz constant of the log barrier. The next iterate for our algorithm is selected by 
\begin{flalign*} 
\alpha &\gets \min \left\{ \frac{\eta}{\| S^{-1} \dir{s} \|_2}, 1 \right\} \\
x^{+} &\gets x + \alpha \dir{x}.
\end{flalign*}
The term $\frac{\eta}{\| S^{-1} \dir{s} \|_2}$ above encourages small step sizes when the linear approximation of the slack variable indicates a large $\alpha$ would cause the algorithm to step outside the feasible region. For example, if we were solving a linear program picking $\eta = 1/2$ would guarantee that $a_i(\xPlus) > a_i(x) / 2 > 0$.

If the predicted progress $\barrierModel_{x} ( \dir{x} )$ is small then \editThree{the algorithm aims to obtain a primal-dual pair corresponding to an approximate Fritz John point}. To do this we need a method for selecting the next dual variable $y^{+}$. An instinctive solution is to pick $y^{+}$ such that $y^{+} = \mu (S^{+})^{-1} \ones$ with $S^{+} = \diag(\cons(x^{+}))$, i.e., a typical primal barrier update. Unfortunately, using this method it is unclear how to construct efficient bounds on $\| \grad_x \Lag(x^{+}, y^{+}) \|_2$. Instead we pick $y^{+}$ using a typical primal-dual step, i.e,
\begin{flalign*}%\label{eq:update-y-plus}
y^{+} \gets y + \dir{y}
\end{flalign*}
where $\dir{y}$ satisfies
\begin{flalign*}%\label{eq:primal-dual-comp}
S \dir{y} + Y \dir{s} + S y = \mu \ones
\end{flalign*}
with $y = \mu S^{-1} \ones$ and $\dir{s} = \grad \cons(x) \dir{x}$. We remark that because $y = \mu S^{-1} \ones$ this can be simplified to $y^{+} \gets \mu S^{-1}  \ones  -  \mu S^{-2} \dir{s}$. Hence, Algorithm~\ref{algIPM} is a hybrid between a traditional primal-dual method and a pure primal method. 

Algorithm~\ref{algIPM} terminates when it reaches an approximate second-order stationary interior point (SIP) which is defined by \eqref{eq:first-order-SIP-full:first-order} and \eqref{eq:first-order-SIP-full:second-order}.

\begin{definition}
A \textbf{$(\mu,\tau_{l}, \tau_{c})$-approximate first-order SIP} satisfies
\begin{taggedsubequations}{SIP1}\label{eq:first-order-SIP-full:first-order}%\mytag{FJ1}
\begin{flalign}
 (\cons(\x), \y) &>  \zeros \label{eq:first-order-SIP-full:feasible} \\ %\mytag{FJ1a} \\
\abs{\y_i \cons_i(\x) - \mu} &\le \frac{\tau_{c} \mu}{2} \quad \forall i \in \ConSet \label{eq:first-order-SIP-full:dual-gap}  \\
\norm{\grad_{x} \Lag(\x,\y) }_{2} &\le \tau_{l} \mu \sqrt{\norm{\y}_{1} + 1}\label{eq:first-order-SIP-full:grad-lag}.
\end{flalign}
\end{taggedsubequations}
\end{definition}

One should interpret \eqref{eq:first-order-SIP-full:first-order} thinking of $\mu \in (0,\infty)$ becoming arbitrarily small, and $\tau_{l} \in (0, \infty)$ as a fixed constant which allows us to trade off how small we want $\norm{\grad_{x} \Lag(x,y) }_{2}$ relative to $y_i \cons_i(x)$. The term $\tau_{l}$ recognizes that the duality gap and dual feasibility are not directly comparable quantities. Additionally, $\tau_{c} \in (0,1\editThree{]}$ is a fixed constant specifying how tightly we want perturbed complementarity to hold.
These first-order optimality conditions are slightly stronger conditions than our earlier definition of an approximate Fritz John point, i.e., \eqref{eq:True-Fritz-John}, because
we can construct a solution to \eqref{eq:True-Fritz-John} from a solution to \eqref{eq:first-order-SIP-full:first-order}.
In particular, it suffices
to solve \eqref{eq:first-order-SIP-full:first-order} with $\tau_{l} \in (0,1]$, $\tau_{c} \in (0,1\editThree{]}$, and set \editTwo{$t = \frac{1}{1 + \| y \|_1}$} and $\lambda = \frac{y}{1 + \| y \|_1}$ to obtain a solution to \eqref{eq:True-Fritz-John}.
\editFour{The reader may also observe that \eqref{eq:True-Fritz-John} is a mix 
of both $\| \cdot \|_2$ and $\| \cdot \|_1$ norms, an explaination for this choice will be provided later in Remark~\ref{remark:mix-of-norms}.}

\begin{definition}\label{def:interior-point-full:second-order}
A \textbf{$(\mu,\tau_{l},\tau_{c})$-approximate second-order SIP} satisfies equation~\eqref{eq:first-order-SIP-full:first-order} and
\begin{flalign}\mytag{SIP2}
\editThree{\grad_{xx}^2 \Lag(\x,\y) +  \mu \grad \cons(\x)^T S^{-2} \grad \cons(\x) \succeq \tau_{c}^{1/2} \LipGrad (1 + \| \y \|_1 ) \eye.}
\label{eq:first-order-SIP-full:second-order}
\end{flalign}
\end{definition}

\editThree{
\editThree{Note if we} \editTwo{consider a sequence $(x\seq{k}, y\seq{k}, \tau_{l}\seq{k}, \tau_{c}\seq{k}, \mu\seq{k})$ satisfying \eqref{eq:first-order-SIP-full:first-order} and \eqref{eq:first-order-SIP-full:second-order}
with $\lim_{k \rightarrow \editFour{\infty}} (\tau_{c}\seq{k}, \mu\seq{k} ) = \zeros$, i.e., assume that we wrap Algorithm~\ref{algIPM} in an outer Algorithm that reduces the termination tolerances $\tau_{c}$ and $\mu$.
Then, \editThree{Lemma~\ref{lem:second-order-necessary-conditions-hold-in-limit} shows} if this sequence \editThree{limits to $(x^{\star}, y^{\star}, \tau_{l}^\star, \zeros)$ where $(x^{\star}, y^{\star})$ is a KKT point} (for example, see Lemma~\ref{lem:MFCQ} and associated discussion) then this limit point also satisfies the second-order necessary conditions \cite[Section~12.4]{nocedal2006numerical}.}

\begin{lemma}\label{lem:second-order-necessary-conditions-hold-in-limit}
If the sequence $(x\seq{k}, y\seq{k}, \tau_l\seq{k}, \tau_c\seq{k}, \mu\seq{k})$ satisfies \eqref{eq:first-order-SIP-full:first-order} and \eqref{eq:first-order-SIP-full:second-order}, \editFour{and} limits to $(x^\star, y^\star, \tau_l^\star, 0, 0)$ then $u^T \grad_{xx}^2 \Lag(\x^{\star},\y^{\star}) u \ge 0$, $\forall u \in U := \{ v \in \ball{1}{\zeros} : \grad a_i(x^{\star})^T v = 0 \quad \forall i \in \mathcal{A} \}$ where $\mathcal{A} := \{ i \in [m] : a_i(\x^\star) = 0 \}$.
\end{lemma}

\begin{myproof}
As $\obj$ and $\cons$ are differentiable, for sufficiently large $k$ there exists some constant $C > 0$ such that
\begin{flalign}\label{eq:non-active-second-order-terms-decay}
\left\| \sum_{i \not\in \mathcal{A}} \frac{\mu\seq{k}}{\cons_i(x\seq{k})^2} \grad \cons_i(\x\seq{k}) \grad \cons_i(\x\seq{k})^T \right\|_2 \le C \mu\seq{k}.
\end{flalign}
Therefore, by \eqref{eq:first-order-SIP-full:second-order} and \eqref{eq:non-active-second-order-terms-decay}, for sufficiently large $k$ we have
\[
 \grad_{xx}^2 \Lag(\x\seq{k},\y\seq{k}) + \sum_{i \in \mathcal{A}} \frac{\mu\seq{k}}{\cons_i(x\seq{k})^2} \grad \cons_i(\x\seq{k}) \grad \cons_i(\x\seq{k})^T \succeq -(C \mu\seq{k} + (\tau_{c}\seq{k})^{1/2} \LipGrad ( 1 + \| y\seq{k} \|_1)) \eye.
\] 
Next, note that $U$ is compact and consider a sequence
$u\seq{k} \in \argmin_{u \in U} u^T \grad_{xx}^2 \Lag(\x\seq{k},\y\seq{k}) u$.
Then, as by assumption $\obj$ and $\cons$ are twice differentiable, we have 
\[ 
\min_{u \in U} u^T \grad_{xx}^2 \Lag(\x^\star,\y^\star) u = \lim_{k \rightarrow \infty} (u\seq{k})^T \grad_{xx}^2 \Lag(\x\seq{k},\y\seq{k}) u\seq{k} \ge \lim_{k \rightarrow \infty} -(C \mu\seq{k} + (\tau_{c}\seq{k})^{1/2} \LipGrad ( 1 + \| y\seq{k} \|_1)) = 0
\]
as desired.
\end{myproof}
}

\begin{remark}\label{remark:solving-trust-region-subproblem}
Our algorithm requires an exact solution to the trust-region subproblem. In an exact arithmetic model of computation, this can be solved exactly in 
$O(n^3)$ arithmetic operations by reducing the problem to a generalized eigenvalue problem \cite[Equation~(34)]{gander1989constrained}. This generalized eigenvalue problem can be solved exactly, e.g., see \cite[Section~7.2]{ghojogh2019eigenvalue} where
\[
\varepsilon=0, \quad A = \begin{pmatrix} -\eye & C \\
C & -\frac{1}{s^2} b b^T \end{pmatrix}, \quad B = \begin{pmatrix} 0 & \eye \\
\eye & 0 \end{pmatrix}
\]
and using that the eigenvalues of $B$ are nonzero. 
Our proofs can also be modified to accept approximate
solutions to the trust-region subproblem, for example, one can replace $\dir{x} \in \argmin_{u \in \ball{r}{\zeros}}{ \barrierModel_{x} ( u ) }$ with
\[
\dir{x} \in \DirectionSet{u} \defeq \left\{ \dir{x} \in \reals^n : \exists \tilde{v} \in \reals^{n} \text{ s.t. } \| \tilde{v} \|_2 \le \frac{\tau_{l} \mu}{10} \text{ and } \dir{x} \in \argmin_{u \in \ball{r}{\zeros}} \barrierModel_{x}(u) + \tilde{v}^T u \  \right\}.
\]
Solutions to this problem can be found using standard
techniques for approximately solving trust-region sub-problems \cite{gould2010solving}.
However, for simplicity of presentation and proofs we assume exact solutions to the trust-region subproblem.
\end{remark}

\section{Lemmas on local approximations and search directions}\label{sec:elementary}

This section develops a series of useful lemmas for analyzing Algorithm~\ref{algIPM}.
We believe these Lemmas will also facilitate analysis
of other nonconvex interior point methods.
 In particular, Section~\ref{sec:approximations} bounds the error of Taylor approximations of several useful quantities as a function of the directions.  Section~\ref{sec:lem:bound-direction-size} proves a key lemma bounding the directions in terms of predicted progress. 

\subsection{The accuracy of local approximations}\label{sec:approximations}

In this subsection, as a function of the directions $\| \dir{x} \|_2$, $\| Y^{-1} \dir{y} \|_2$ and $\| S^{-1} \dir{s} \|_2$, we bound the following quantities.

\begin{enumerate}
\item The gap between the predicted reduction and the actual reduction of the log barrier (Lemma~\ref{lemPrimalDualApproxResultTwo}). This allows us to convert predicted reduction $\barrierModel_x(\dir{x})$ into a reduction in the log barrier.
\item Perturbed complementarity $\abs{a_i(x^{+}) y_i^{+} - \mu}$ (Lemma~\ref{lem:complementary}). This allows us to establish when \eqref{eq:first-order-SIP-full:dual-gap} holds.
\item The norm of the gradient of the Lagrangian (Lemma~\ref{lem:LagError}). This allows us to establish when \eqref{eq:first-order-SIP-full:grad-lag} holds. Therefore Lemma~\ref{lem:complementary} and \ref{lem:LagError} allow us to reason about when we are at an approximate first-order SIP.
\end{enumerate}

Globally the log barrier does not have Lipschitz second derivatives. But Lemma~\ref{lem:log-g-bound} shows it is possible to bound the Lipschitz constant of second derivatives of $\log(g(\theta))$ in a neighborhood of the current point. 

\begin{restatable}{lemma}{lemBoundLogG}\label{lem:log-g-bound}
Suppose the function $g : \reals \rightarrow \reals$ has $\LipGrad$-Lipschitz first derivatives and $\LipHess$-Lipschitz second derivatives on the set $[0, \theta]$ where $\theta \in \Rp$. Further assume $g(0) > 0$, $\beta \in (0,1/4]$, and the inequality $\frac{| \theta g'(0) |}{g(0)} + \frac{\LipGrad \theta^2}{g(0)}  \le \beta$ holds. Then $\frac{g(\theta)}{g(0)} \in [\frac{3}{4}, \frac{4}{3} ]$ and $\theta^3 \abs{\frac{\partial^3 \log(g(\theta))}{\partial^3 \theta}} \le  \frac{2 \LipHess \theta^3 + \editThree{8} \LipGrad  \theta^2 \beta}{g(0)} + 5 \beta^3$.
\end{restatable}

\begin{myproof}
	We have
	\begin{flalign*}%\label{eq:bound-g-zero-theta}
	\frac{ \abs{g(0) - g(\theta)}}{g(0)} \le \frac{| \theta g'(0) |}{g(0)} + \frac{\LipGrad \theta^2}{2 g(0)} \le \beta \le \frac{1}{4}.
	\end{flalign*}
	The first inequality uses $\abs{g(0) + g'(0) \theta - g(\theta)} \le \frac{\LipGrad \theta^2}{2}$, the triangle inequality and $g(0) > 0$. The second and third inequality follows from the assumed bound in the theorem statement. Therefore we have established $\frac{g(\theta)}{g(0)} \in [3/4,4/3]$.
	
	We turn to proving our bound on the third derivatives of $\log(g(\theta))$,
	\begin{flalign}
		\notag \frac{\partial \log(g(\theta))}{\partial \theta} &= \frac{g'(\theta)}{g(\theta)} \\
		\notag \frac{\partial^2 \log(g(\theta))}{\partial^2 \theta} &=  \frac{g''(\theta)}{g(\theta)} - \frac{g'(\theta)^2}{g(\theta)^2} \\
		\frac{\partial^3 \log(g(\theta))}{\partial^3 \theta} &= \frac{g'''(\theta)}{g(\theta)} - \frac{3 g'(\theta) g''(\theta)}{g(\theta)^2} + 2 \frac{g'(\theta)^3}{g(\theta)^3}. \label{eq:third-der-log-g}
	\end{flalign}
	By \eqref{eq:third-der-log-g}, $\frac{g(\theta)}{g(0)} \in [3/4,4/3]$, $\abs{g'''(\theta)} \le \LipHess$, and $\abs{g''(\theta)} \le \LipGrad$ we have
	$$
	\abs{\frac{\partial^3 \log(g(\theta))}{\partial^3 \theta}} \le (4/3) \frac{\LipHess}{g(0)} + 3 (4/3)^2 \frac{\LipGrad \abs{g'(\theta)}}{g(0)^2} +  2 (4/3)^3 \frac{\abs{g'(\theta)}^3}{g(0)^3}.
	$$
	Now,
	\editThree{multiplying the previous inequality by $\theta^3$, using $\abs{g'(\theta) - g'(0)} \le \LipGrad \theta$, and the triangle inequality gives}
	\begin{flalign*}
		\theta^3 \abs{\frac{\partial^3 \log(g(\theta))}{\partial^3 \theta}} &\le  (4/3)  \frac{\LipHess \theta^3}{g(0)} + 3 (4/3)^2 \frac{\LipGrad \theta^2 (  \editThree{\abs{\theta g'(0)}} + \LipGrad \theta^2 )}{g(0)^2} + 2 (4/3)^3  \frac{(\abs{\theta g'(0)} + \LipGrad \theta^2)^3}{g(0)^3}  \\
		&\le \frac{2 \LipHess \theta^3 + \editThree{8} \LipGrad  \theta^2 \beta}{g(0)} + 5 \beta^3.
	\end{flalign*}
\end{myproof}

\CHECKED

Lemma~\ref{lem:log-g-bound} only gives us a bound on the local Lipschitz constant for the second derivatives of $\log(g(\theta))$ when $g$ is univariate. By applying Lemma~\ref{lem:log-g-bound} with $g(\theta) := \cons_i(x + \theta v)$, $v = \frac{\dir{x}}{\| \dir{x} \|_2}$ we can bound the difference between the actual and predicted progress on the log barrier function. This bound is given in Lemma~\ref{lemPrimalDualApproxResultTwo}.

\begin{restatable}{lemma}{lemPrimalDualApproxResultTwo}\label{lemPrimalDualApproxResultTwo}
Suppose Assumption~\ref{assume-lip-deriv} holds (Lipschitz derivatives). Let $x \in \X$, $S = \diag(\cons(x))$, $\dir{x} \in \reals^{\NumVar}$,  $\dir{s} = \grad \cons(x) \dir{x}$, $y = \mu S^{-1} \ones$, and $\kappa \in (0, 1/4]$. If
\begin{flalign}\label{eq:kappa-requirement}
\| S^{-1} \dir{s} \|_{2} + \frac{\LipGrad \| \dir{x} \|_2^2  \| y \|_2}{\mu} \le \kappa,
\end{flalign}
then \editThree{$\Convex\{ x, x + \dir{x} \} \subseteq \X$} and
\begin{flalign*}
\left| \barrier(x) + \barrierModel_{x} (\dir{x}) -  \barrier(x + \dir{x}) \right| \le \frac{\LipHess}{6} \left( 1 +  2 \| y \|_1 \right) \| \dir{x} \|_2^3  + \editThree{\frac{4}{3}} \LipGrad \| \dir{x} \|_2^2 \| y \|_1 \kappa + \editThree{\frac{5}{3}} \mu \kappa^3.
\end{flalign*}
\end{restatable}
\CHECKED

\begin{myproof}
	First, \editFour{we aim to prove $\Convex\{ x, x + \dir{x} \} \subseteq \X$}. Define $v \defeq \dir{x} / \| \dir{x} \|_2$, $g_i(\theta) \defeq a_i(\x + \theta v)$, 
	\editThree{and 
	\[ 
	\mathcal{F} := \left\{ \hat{\theta} \in [0,\infty) : \forall \theta \in [0, \hat{\theta}], \frac{g_i(\theta)}{g_i(0)} \in [3/5,5/3] \right\}
	\]
	Note $0 \in \mathcal{F}$ so \editThree{$\theta^\star := \sup_{\hat{\theta} \in [0,\| \dir{x} \| ]} \mathcal{F}$} is well-defined. Since $\cons_i$ is a continuous function it follows that $\mathcal{F}$ is  a closed set and thus $\theta^\star \in \mathcal{F}$. 
	Using that $a_i(\x)$ has $\LipGrad$-Lipschitz first derivatives and $\LipHess$-Lipschitz second derivatives on the set $\X$ we deduce that $g_i(\theta)$ satisfies the same properties on the set $[0,\theta^\star]$. Applying Lemma~\ref{lem:log-g-bound}, \eqref{eq:kappa-requirement} and $\theta^\star \in [0,\| \dir{x} \| ]$ we deduce 
	\[ 
	\frac{\cons_i(x + \theta^\star v)}{\cons_i(x)} = \frac{g_i(\theta^\star)}{g_i(0)} \in [3/4,4/3].
	\]
	This implies $\theta^{\star} = \| \dir{x} \|_2$ since otherwise we could construct a $\theta \in \mathcal{F} \cap [0,\| \dir{x} \|_2]$ with $\theta > \theta^{\star}$ using that $\cons_i$ and therefore $g_i$ is continuous (note the use of the wider interval  $[3/5,5/3]$ instead of $[3/4,4/3]$ in the construction of $\mathcal{F}$). 
	We conclude $\Convex\{ x, x + \dir{x} \} \subseteq \X$.}
	
	Before bounding $\left| \barrier(x) + \barrierModel_{x} (\dir{x}) -  \barrier(x + \dir{x}) \right| $ we provide some auxiliary bounds. Define
	$$
	\beta_i :=\frac{|  \grad \cons_i(x)^T \dir{x} |}{\cons_i(x)} + \frac{\LipGrad \| \dir{x} \|_2^2}{\cons_i(x)},
	$$
	for all $i \in \ConSet$. Then we have,
	\begin{flalign*}
		\| \beta \|_2^2  &=  \sum_{i=1}^{\NumCon} \left( \frac{|  \grad \cons_i(x)^T \dir{x} |}{\cons_i(x)} + \frac{\LipGrad \| \dir{x} \|_2^2 y_i}{\mu}  \right)^2 \notag \\ 
		&\le 2 \sum_{i=1}^{\NumCon}  \left( \left( \frac{|  \grad \cons_i(x)^T \dir{x} |}{\cons_i(x)} \right)^2 + \left( \frac{\LipGrad \| \dir{x} \|_2^2}{\mu}  \right)^2 y_i^2  \right) \notag \\
		&= 2 \| S^{-1} \dir{s} \|_2^2 + 2 \left( \frac{\LipGrad \| \dir{x} \|_2^2}{\mu}  \right)^2 \| y \|_2^2 \notag \\
		&\le 2 \kappa^2
	\end{flalign*}
	where the first equality uses $1 / \cons_i(x) = y_i / \mu$, the first inequality uses the fact that $(a+b)^2 \le 2 (a^2 + b^2)$, and the final inequality uses $a^2 + b^2 \le (a+b)^2$ for $a, b \ge 0$. Hence,
	\begin{flalign}
		\sum_{i=1}^{\NumCon} \beta_i^3 \le \| \beta \|_2^2 \max_{i \in \ConSet}\{ \beta_i \} \le \| \beta \|_2^2 \kappa \le 2 \kappa^3 \label{eq:beta-cubed-bound}
	\end{flalign}
	\editThree{where the second inequality uses $\beta_i \le \kappa$ because $\beta_i \le \| S^{-1} \grad \cons (x) \dir{x} \|_{\infty} + \LipGrad \| \dir{x} \|_2^2 \| S^{-1} \ones \|_{\infty} = \| S^{-1} d_s \|_{\infty} + \frac{\LipGrad \| \dir{x} \|_2^2 \| y \|_{\infty}}{\mu} \le \kappa$.}
	Observe, also by Taylor's Theorem and the fact that \editThree{$\grad^2 \obj$} is Lipschitz on $\X$ that
	\begin{flalign}\label{eq:f-taylor}
		\abs{ \obj(x) + \frac{1}{2} \dir{x} \grad^2 \obj(x) \dir{x} + \grad \obj(x)^T \dir{x} - \obj(x+\dir{x}) } \le \frac{\LipHess}{6} \| \dir{x} \|_2^3.
	\end{flalign}
	Using Lemma~\ref{lem:log-g-bound} and Taylor's Theorem with $g_i(\theta) \defeq \cons_i(x + \theta v)$, $h_i(\theta) \defeq \log(g_i(\theta))$, and $v = \frac{\dir{x}}{\| \dir{x} \|_2}$, we get
	\begin{flalign}\label{eq:tmp139535}
		\abs{h_i(0) + \theta h_i'(0) + \frac{\theta^2}{2} h_i''(0) - h_i(\theta) } 
		\le \frac{\theta^3}{6} \sup_{\hat{\theta} \in [0,\theta]} h_i'''(\hat{\theta}) 
		\le \frac{1}{6} \left(  \frac{2 \LipHess \theta^3 + \editThree{8} \LipGrad  \theta^2 \beta_i}{g(0)} + 5 \beta_i^3 \right).
	\end{flalign}
	We can now bound the quality of a second-order Taylor series expansion of $\barrier$ as
	\begin{flalign*}
		\left| \barrier(x) + \barrierModel_{x} (\dir{x}) -  \barrier(x + \dir{x}) \right| &\le \frac{\LipHess}{6} \| \dir{x} \|_2^3 + \mu \sum_{i=1}^{\NumCon} \left(  \frac{2 \LipHess \| \dir{x} \|_2^3 + \editThree{8} \LipGrad \| \dir{x} \|_2^2 \beta_i}{6 a_i(x)}  + \frac{5 \beta^3_i}{6} \right) \\
		%%%%%%%
		&\le  \frac{\LipHess}{6} \| \dir{x} \|_2^3 + \sum_{i=1}^{\NumCon} \left(  y_i  \left( \frac{\LipHess \| \dir{x} \|_2^3}{3} + \editThree{\frac{4}{3}} \LipGrad \| \dir{x} \|_2^2 \beta_i  \right) + \mu \frac{5\beta^3_i}{6} \right) \\
		%%%%%%%
		&\le  \frac{\LipHess}{6} \left( 1 + 2 \| y \|_1 \right) \| \dir{x} \|_2^3  + \editThree{\frac{4}{3}} \LipGrad \| \dir{x} \|_2^2 \| y \|_1 \kappa + \editThree{\frac{5}{3}} \mu \kappa^3.
	\end{flalign*}
	The first inequality uses \eqref{eq:f-taylor} and \eqref{eq:tmp139535}. The second inequality uses $1 / \cons_i(x) = y_i / \mu$. The third inequality uses $\beta_i \le \kappa$ and \eqref{eq:beta-cubed-bound}.
	% and \eqref{eq:beta-y-bound}.
\end{myproof}

% One can see from \eqref{eq:kappa-requirement} if we wish to use Lemma~\ref{lemPrimalDualApproxResultTwo} to guarantee we remain feasible we must select $\| \dir{x} \|$ such that $\frac{\LipGrad (1 + \| y  \|_{1}) \| \dir{x} \|_2^2}{\mu} = \bigO{1}$.

Observe that if \eqref{eq:kappa-requirement} holds for some $x \in \X$ and $\dir{x}$ then \eqref{eq:kappa-requirement} holds for any damped direction $\alpha \dir{x}$ with $\alpha \in [0,1]$, i.e., $\Convex\{ x, x + \alpha \dir{x} \} \subseteq \Convex\{ x, x + \dir{x} \} \subseteq \X$. This observation ensures we can use Lemma~\ref{lemPrimalDualApproxResultTwo} to establish the premises of Lemma~\ref{lem:complementary} and \ref{lem:LagError} which require $\Convex\{\x, \xPlus\} \subseteq \X$.

\begin{restatable}{lemma}{lemComplementary} \label{lem:complementary}
Suppose Assumption~\ref{assume-lip-deriv} holds. Let $\Convex\{\x, \xPlus\} \subseteq \X$, $s = \cons(\x)$, $s^{+} = a(\xPlus)$, $S = \diag(\cons(x))$, $Y = \diag(y)$, $\yPlus \in \reals^{\NumCon}$, $Y^{+} = \diag(\yPlus)$, $\dir{x} = \xPlus - x$, $\dir{y} = \yPlus - y$, and $\dir{s} = \grad \cons(x) \dir{x}$.
If the equation $S y + S \dir{y} + Y \dir{s} = \mu \ones$ holds, then
\begin{flalign}
\| Y^{-1} \dir{y} \|_{2} &\le \| S^{-1} \dir{s} \|_{2} + \|  \mu (S Y)^{-1} \ones - \ones \|_{2} \label{eq:Y-bound} \\
\| Y^{+} s^{+} - \mu \ones \|_{2} &\le  \| S y \|_{\infty} \| S^{-1} \dir{s} \|_{2} \| Y^{-1} \dir{y} \|_{2} + \frac{\LipGrad}{2}  \| y \|_{2} (1 + \| Y^{-1} d_{y} \|_{2}) \| \dir{x} \|_2^2.
\label{eq:new-general-comp-bound}
\end{flalign}
Furthermore, if $\| Y^{+} s^{+} - \mu \ones \|_{\infty} < \mu$ and $ \| Y^{-1} \dir{y} \|_{\infty} \le 1$ then $s^{+}, y^{+} \in \Rpp^{\NumCon}$.
\end{restatable}

\ifRepeatThms
\lemComplementary*
\fi

\begin{myproof}
	To show \eqref{eq:Y-bound} notice that multiplying $Sy + S \dir{y} + Y \dir{s} = \mu \ones$ by $(S Y)^{-1}$ and rearranging yields
	$Y^{-1} \dir{y} =  - S^{-1} \dir{s} + \editThree{( \mu (S Y)^{-1} \ones  - \ones)}$.
	
	Next, we show \eqref{eq:new-general-comp-bound}.
	Observe that
	\begin{flalign}
		\notag s^{+}_i y^{+}_i - \mu  &= a_i(x + \dir{x}) (y_i + d_{y_i}) - \mu \\
		\notag &= (d_{s_i} + a_i(x)) (y_i + d_{y_i}) + (a_i(x + \dir{x}) - (d_{s_i} + a_i(x) ) ) (y_i + d_{y_i}) - \mu \\
		&= d_{s_i} d_{y_i} + (a_i(x + \dir{x}) - (d_{s_i} + a_i(x) ) ) (y_i + d_{y_i}), \label{eq:comp-09324}
	\end{flalign}
	where the first transition is by definition of $s_i^{+}$ and $y_i^{+}$, the second transition comes from adding and subtracting $(d_{s_i} + a_i(x)) (y_i + d_{y_i})$, and the third transition by substituting $\mu = s_i y_i + s_i d_{y_i} + y_i d_{s_i} = a_i(x) y_i + a_i(x) d_{y_i} + y_i d_{s_i}$. Furthermore, since $\grad a_i$ is $\LipGrad$-Lipschitz continuous on $\X$,
	$$
	\abs{ a_i(x + \dir{x})  - (d_{s_i} + a_i(x)) } = \abs{ a_i(x + \dir{x}) - (\grad a_i(x) \dir{x} + a_i(x)) } \le \frac{\LipGrad}{2} \| \dir{x} \|^2_2,
	$$
	combining this equality with \eqref{eq:comp-09324} yields
	$$
	\abs{ s^{+}_i y^{+}_i - \mu} \le \abs{ d_{s_i} d_{y_i} } + \frac{\LipGrad}{2}  y^{+}_i \| \dir{x} \|_2^2 \le \abs{s_i y_i} \abs{ s^{-1}_i d_{s_i}} \abs{y^{-1}_i d_{y_i} } + \frac{\LipGrad}{2}  y_i ( 1 + y_i^{-1} \dir{y_i}) \| \dir{x} \|_2^2.
	$$
	We deduce \eqref{eq:new-general-comp-bound} by Cauchy-Schwarz. 
	The fact that $y^{+} \in \Rp^{\NumCon}$ follows from $\| Y^{-1} \dir{y} \|_{\infty} \le 1$. The fact that $y^{+}, s^{+} \in \Rpp^{\NumCon}$ follows from $y^{+} \in \Rp^{\NumCon}$  and $\| S^{+} y^{+} - \mu \|_{\infty} < \mu$.
\end{myproof}

\CHECKED

 Lemma~\ref{lem:complementary} will allow us to guarantee $(x^{+}, y^{+})$ satisfies \eqref{eq:first-order-SIP-full:feasible} and \eqref{eq:first-order-SIP-full:dual-gap} when we take a primal-dual step in Algorithm~\ref{algIPM}. This a typical Lemma used for interior point methods in linear programming except that the nonlinearity of the constraints creates the additional $\frac{\LipGrad}{2}  \| y \|_{2} (1 + \| Y^{-1} d_{y} \|_{2}) \| \dir{x} \|_2^2$ term in \eqref{eq:new-general-comp-bound}.

\begin{restatable}{lemma}{lemLagError}\label{lem:LagError}
Suppose Assumption~\ref{assume-lip-deriv} holds.
Let $y, \yPlus \in \reals^{\NumCon}$ and $\Convex\{\x, \xPlus\} \subseteq \X$. Then the following inequality holds:
\begin{flalign}
&\| \grad_{x} \Lag( x , y ) + \grad_{xx}^2 \Lag (x, y )^T \dir{x} - \dir{y}^T  \grad_{x} \cons(x)   - \grad_{x} \Lag(x^{+}, y^{+}) \nonumber
 \|_{2} \\
 &\le \LipGrad \| y \|_2  \| \dir{x} \|_2 \| Y^{-1} \dir{y} \|_2 + \frac{\LipHess}{2}  (\| y \|_1 + 1) \| \dir{x} \|^2_2
\end{flalign}
with $\dir{x} = \xPlus - x$ and $\dir{y} = \yPlus - y$.
\end{restatable}
\CHECKED

\begin{myproof}
	Observe that:
	\begin{flalign*}
		& \left\| \sum_{i \in \ConSet}{  \left( y_i  \grad a_i(x) + y_i  \grad^2 a_i(x) \dir{x}  - \dir{y_i}  \grad a_i(x) - y_i^{+} \grad a_i(x^{+}) \right) }  \right\|_2  \\
		%%%%%%%%%%%%%%%%%%
		&\le\sum_{i \in \ConSet}{  \left\| y_i  \grad a_i(x) + y_i  \grad^2 a_i(x) \dir{x}  + \dir{y_i}  \grad a_i(x) - y_i^{+} \grad a_i(x^{+})  \right\|_2 }  \\
		%%%%%%%%%%%%%%%%%%
		&\le \editThree{\sum_{i \in \ConSet}}{\editThree{y_i}  \left\| \grad a_i(x)  +  \grad^2 a_i(x) \dir{x}   -  \grad a_i(x^{+}) \right\|_{2} +  \editThree{\dir{y_i}}  \left\| \grad a_i(x) - \grad a_i(x^{+}) \right\|_{2}} \\
		%%%%%%%%%%%%%%%%%%
		&\le  \frac{\LipHess}{2} \| y \|_1 \| \dir{x} \|^2_2 + \LipGrad \| \dir{y} \|_1 \| \dir{x} \|_2,
	\end{flalign*}
	where the first and second transition hold by the triangle inequality, the third transition applying \eqref{eq:taylor} using the Lipschitz continuity of $\grad a$ and $\grad^2 a$. Next, by the triangle inequality, the inequality we just established, and Taylor's theorem with Lipschitz continuity of $\grad f$ we get
	\begin{flalign}
		\notag &\| \grad_{x} \Lag( x , y ) + \grad_{xx}^2 \Lag (x, y )^T \dir{x} - \dir{y}^T \grad_{x} \cons(x)   - \grad_{x} \Lag(x^{+}, y^{+}) 
		\|_{2} \\
		\notag &\le \left\|  \grad f(x) + \grad^2 f(x) \dir{x} - \grad f(x^{+}) \right\|_{2} + \left\| \sum_{i \in \ConSet} \left( y_i  \grad a_i(x) + y_i  \grad^2 a_i(x) \dir{x}  + \dir{y_i}  \grad a_i(x) - y_i^{+} \grad a_i(x^{+})  \right)  \right\|_2 \\
		%+ \LipHess \| y \|_1 \| x^{+} - x \|^2_2 + \LipGrad \| \dir{y} \|_1 \| \dir{x} \|_2 \\
		&\le \frac{\LipHess}{2} (\| y \|_1 + 1) \| \dir{x} \|^2_2 + \LipGrad \| \dir{y} \|_1 \| \dir{x} \|_2.
	\end{flalign}
\end{myproof}

Lemma~\ref{lem:LagError} allows us to guarantee that \eqref{eq:first-order-SIP-full:grad-lag} holds at $(x^{+}, y^{+})$  when $\| \dir{x} \|_2$ and $\| Y^{-1} \dir{y} \|_2$ are small. The introduction of the $\LipGrad \| y \|_2  \| \dir{x} \|_2 \| Y^{-1} \dir{y} \|_2$ term is the key reason that the analysis of \cite{bian2015complexity,haeser2019optimality,ye1992new} for affine scaling does not automatically extend into nonlinear constraints because it does not efficiently bound $\| Y^{-1} \dir{y} \|_2$.

\begin{remark}\label{remark:mix-of-norms}
The reader might observe that our termination criteria \eqref{eq:first-order-SIP-full:first-order} has a strange mix of norms, in particular the size of $\grad_x \Lag(x,y)$ is measured using $\| \cdot \|_2$ and the the size of $y$ is measured by $\| \cdot \|_1$. We attempt to explain this by showing how these norms naturally appear in the Lemmas in this section. The bound on $\| \grad_{x} \Lag( x , y ) + \grad_{xx}^2 \Lag (x, y )^T \dir{x} - \dir{y}^T  \grad_{x} \cons(x)   - \grad_{x} \Lag(x^{+}, y^{+})  \|_{2}$  in Lemma~\ref{lem:LagError} contains a term of the form $\frac{\LipHess}{2} \| y \|_1 \| \dir{x} \|_2$. \editThree{This term is tight because if we select $f(x) := 0$, $a_i(x) := \frac{\LipHess}{6} (v^T x)^3 + 1$ for some $v$ with $\| v \|_2 = 1$, and then consider $x = \zeros$, $\dir{x} = \theta v$ for some $\theta \in (0,\infty)$, and $\dir{y} = \zeros$} then $\| \grad_{x} \Lag( x , y ) + \grad_{xx}^2 \Lag ( x, y )^T \dir{x} - \dir{y}^T  \grad_{x} \cons(x)   - \grad_{x} \Lag(\dir{x}, \dir{y})  \|_{2} = \| \grad_{x} \Lag(\dir{x}, y) \|_2 = \| \sum_{i \in \ConSet} y_i \frac{\LipHess}{2} (v^T \dir{x})^2 v \|_2 =  \frac{\LipHess}{2} (v^T \dir{x})^2 \| y \|_1 = \frac{\LipHess}{2} \| y \|_1 \| \dir{x} \|_2^2  $. Furthermore, one can see from this example that changing the norm of $\| y \|_1$ would introduce a dimension-factor and make the bound strictly weaker. Trust-region subproblems can be efficiently solved when $\dir{x}$ is bounded in Euclidean norm. For this reason, we choose to use the Euclidean norm to measure the size of $\dir{x}$. Inspection of the proof of Lemma~\ref{lem:LagError} indicates that one cannot change the norm on the term $\grad_{x} \Lag( x , y ) + \grad_{xx}^2 \Lag (x, y )^T \dir{x} - \dir{y}^T  \grad_{x} \cons(x)   - \grad_{x} \Lag(x^{+}, y^{+})$ without changing the norm on the term $\dir{x}$ or introducing a dimension-factor. For similar the reasons it is inadvisable to change the norms on the term $\frac{\LipHess}{3} \| y \|_1 \| \dir{x} \|_2^3$ in Lemma~\ref{lemPrimalDualApproxResultTwo}.
\end{remark}

\subsection{Bounding the direction of the slack variables}\label{sec:lem:bound-direction-size}

This section presents Lemma~\ref{lem:SY-bound} which allows us to bound the direction of the slack variables.
Before proving Lemma~\ref{lem:SY-bound} we state Lemma~\ref{lem:trust-region-facts} which contains some basic and well-known facts about trust-region subproblems that will be useful. 

\newcommand{\deltar}[0]{\editThree{\delta(r)}}

\begin{restatable}{lemma}{lemTrustRegionFacts}\label{lem:trust-region-facts}
Consider $g \in \reals^{\NumVar}$ and a symmetric matrix $H \in \reals^{\NumCon \times \NumVar}$.
Define $\Delta(u) := \frac{1}{2} u^T H u + g^T u$ where $\Delta : \reals^{\NumVar} \rightarrow \reals$ and let $u^{*} \in \argmin_{u \in \ball{r}{\zeros}}{ \Delta(u) }$ be an optimal solution to the trust-region subproblem for some $r \ge 0$. Then
there exists some $\deltar \ge 0$ such that:
\begin{flalign}\label{eq:trust-region-delta-gain}
\deltar (\| u^{*} \|_{2} - r) = 0, ~~ (H + \deltar \eye) u^{*} = -g, \text{~~and~~} H + \deltar \eye \succeq 0.
\end{flalign}
Conversely, if $u^{*}$ satisfies \eqref{eq:trust-region-delta-gain} then $u^{*} \in \argmin_{u \in \ball{r}{\zeros}}{ \Delta(u) }$.
Let $\sigma(r) :=  \min_{u \in \ball{r}{\zeros}}{ \Delta(u) }$, then for all $r \in [0, \infty)$ we have
\begin{subequations}
\begin{flalign}\label{eq:delta-bound:lem:trust-region-facts}
\sigma(r) &\le -\frac{\deltar r^2}{2} \\
\label{eq:trust-region-facts:sigma}
\sigma(r) \le \sigma(\alpha r) &\le \alpha^2 \sigma(r)~~ \forall \alpha \in [0,1].  
\end{flalign}
\end{subequations}
Furthermore, the function $\sigma(r)$ is monotone decreasing and continuous.
\end{restatable}

\begin{myproof}
	Equation~\eqref{eq:trust-region-delta-gain} follows from the KKT conditions, see \citet[Lemma 2.4.]{sorensen1982newton}, \citet[Corollary 7.2.2]{conn2000trust} or \citet[Theorem 4.3.]{nocedal2006numerical}. We now show \eqref{eq:delta-bound:lem:trust-region-facts}. Substituting $(H + \deltar \eye) u^{*} = -g$ into $\frac{1}{2} (u^{*})^T H u^{*} +  g^T u^{*}$ yields $\sigma(r) = \Delta (u^{*}) = 1/2 g^T u^{*} - \deltar /2 \| u^{*} \|^2  \le -\deltar /2 \| u^{*} \|_2^2 $ where the last inequality follows from $g^T u^{*} =  - g^T (H + \deltar \eye)^{-1} g \le 0$. Since \eqref{eq:trust-region-delta-gain} states that either $\deltar = 0$ or $\| u^{*} \|_2 = r$ we conclude \eqref{eq:delta-bound:lem:trust-region-facts} holds. The inequality $\sigma(\alpha r) \le \alpha^2 \sigma(r)$ holds since $\sigma(\alpha r) \le \Delta(\alpha u^{*}) = \frac{1}{2} \alpha^2 (u^{*})^T H u^{*} + \alpha g^T u^{*} \le \frac{1}{2} \alpha^2 (u^{*})^T H u^{*} + \alpha^2 g^T u^{*} = \alpha^2 \sigma(r)$ where the inequality uses $g^T u^{*} \le 0$. The inequality $\sigma(r) \le \sigma(\alpha r)$ holds since any solution to \editThree{$\| u \|_{2} \le \alpha r$ is feasible to $\| u \|_{2} \le r$}. The fact that $\sigma(r)$ is monotone decreasing and continuous follows from \eqref{eq:trust-region-facts:sigma}.
\end{myproof}

Lemma~\ref{lem:SY-bound}, which follows, is key to our result, because it allows us to bound the size of $\| S^{-1} \dir{s} \|_2$ (recall $\dir{s} = \grad \cons(x) \dir{x}$). We remark that often in linear programming one shows $\| S^{-1} \dir{s} \|_2 = \bigO{1 }$ to prove an $\mathcal{O}(\sqrt{n} \log(1/\mu) )$ iteration bound for interior point methods \cite[Lemma 4]{kojima1989polynomial}.
Lemma~\ref{lem:SY-bound} is inspired by this idea from linear programming.
 Combining Lemma~\ref{lem:SY-bound} with the Lemmas from Section~\ref{sec:approximations} allows us to give concrete bounds on the reduction of the log barrier at each iteration. This underpins our main results in Section~\ref{sec:worst-case}.

\begin{restatable}{lemma}{lemSYbound}\label{lem:SY-bound} 
%Suppose Assumption~\ref{assume-lip-deriv} holds.  Let $\mu \in (0, \infty)$. 
Consider $A \in \reals^{\NumCon \times \NumVar}$, $g \in \reals^{\NumVar}$, and a symmetric matrix $H \in \reals^{\NumCon \times \NumVar}$.
Define $\Delta(u) := \frac{1}{2} u^T (H + A^T A) u + g^T u$ where $\Delta : \reals^{\NumVar} \rightarrow \reals$ and let $\dir{x} \in \argmin_{u \in \ball{r}{\zeros}}{ \Delta(u) }$ for some $r \ge 0$. Then
\begin{flalign}\label{eq:S-bound}
\| A \dir{x} \|_2 \le \sqrt{-\dir{x}^T H \dir{x} - 2 \Delta(\dir{x})}.
\end{flalign}
\end{restatable}
\begin{myproof}
Observe that
\begin{flalign*}
\Delta (\dir{x}) &=  \frac{1}{2} \dir{x}^T (H + A^T A) \dir{x} + g^T \dir{x} \\
&= \frac{1}{2} \dir{x}^T  (H + A^T A)  \dir{x} - \dir{x}^T (H + A^T A+ \delta \eye ) \dir{x} \\
&= -\frac{1}{2} \dir{x}^T \left(  H + A^T A \right) \dir{x}  - \delta \| \dir{x} \|_2^2
\end{flalign*}
where the second transition use the fact from Lemma~\ref{lem:trust-region-facts} that there exists some $\delta$ such that $(H + A^T A + \delta \eye ) \dir{x} = -g$. Rearranging this expression and using $\delta \| \dir{x} \|_2^2 \ge 0$ yields %and using that $\grad f$ and $\grad a$ are Lipschitz yields
\begin{flalign}
\| A \dir{x} \|^2_2 \le - \dir{x}^T H \dir{x} - 2 \Delta (\dir{x}).
\end{flalign}
This concludes the proof of Lemma~\ref{lem:SY-bound}.
\end{myproof}

Now, if we set $S = \diag(a(x))$, \editThree{$y = \mu S^{-1} \ones$}, $H = \grad_{xx}^2 \Lag(x,y)$, $A = \editThree{\sqrt{\mu}} S^{-1} \grad \cons(x)$, $\dir{s} = \grad \cons(x) \dir{x}$, \editThree{$g = \grad \barrier(x)$ and $\dir{x} \in \argmin_{u \in \ball{r}{\zeros}}\barrierModel_{x} (u)$, i.e., as per Algorithm~\ref{algIPM}, then 
\[ 
H + A A^T = \grad_{xx}^2 \Lag(x,y) + \mu \grad \cons(x)^T S^{-2} \grad \cons(x) = \grad^2 \barrier(x)
\] 
}
and we deduce from Lemma~\ref{lem:SY-bound} that
\begin{flalign*}
\| S^{-1} \dir{s} \|_2 \editThree{=\frac{1}{\sqrt{\mu}} \| A \dir{x} \|_2 \le \sqrt{\frac{-\dir{x}^T H \dir{x} - 2 \Delta(\dir{x})}{\mu}} = } \sqrt{\frac{-\dir{x}^T \grad_{xx}^2 \Lag(x,y) \dir{x} - 2 \barrierModel_{x} (\dir{x})}{\mu}}.
\end{flalign*}
Moreover, if $\| \grad^2 \obj(x) \|_2 \le \LipGrad$ and $\| \grad^2 a_i(x) \|_2 \le \LipGrad$ then
\begin{flalign}\label{eq:nonconvex-s-bound}
\| S^{-1} \dir{s} \|_2 \le \sqrt{\frac{\LipGrad (1 + \| y \|_1) \| \dir{x} \|_2^2 - 2 \barrierModel_{x} (\dir{x})}{\mu}}.
\end{flalign}
We emphasize that \eqref{eq:nonconvex-s-bound} is unusual because the bound on $\| S^{-1} d_{s} \|_{2}$ depends on the predicted progress for a step size of $\alpha = 1$, i.e., $\barrierModel_{x} (\dir{x} )$. This relates to the importance of adaptive step size selection \editFour{on line~\ref{line:alpha-choice} of} Algorithm~\ref{algIPM} for proving our convergence bounds. The intuition is as follows. At each iteration, if we have not terminated then we aim to reduce the barrier function. Lemma~\ref{lemPrimalDualApproxResultTwo} implies for sufficiently small $\alpha$ that the new point $x + \alpha \dir{x}$ will reduce the barrier function proportional to $\barrierModel_{x} (\alpha \dir{x} )$. If $\| S^{-1} \dir{s} \|_2$ is small then we can take a step size with $\alpha = 1$ and reduce the barrier function proportional to $\barrierModel_{x} (\dir{x} )$. On the other hand, if $\| S^{-1} \dir{s} \|_2$ is big we must pick $\alpha$ small to guarantee that we reduce the barrier function proportional to $\barrierModel_{x} (\alpha \dir{x} )$. \editFour{Since $\alpha$ is small and $\dir{x} \in \argmin_{u \in \ball{r}{\zeros}}\barrierModel_{x} (u)$, $\barrierModel_{x}(\alpha \dir{x} )$ is smaller than $\barrierModel_{x} (\dir{x} )$}. 
Fortunately, this is counterbalanced because if $\| S^{-1} \dir{s} \|_2$ is large that implies using \eqref{eq:nonconvex-s-bound} that $\barrierModel_{x} (\dir{x} )$ is also large.

\section{Iteration bounds for finding approximate Fritz John points}\label{sec:worst-case}

\newcommand{\assumeNC}{A3}

\def\RadUnscaled{\kappa^{3/4}}

\def\RadNCval{\RadNC(y)}
\def\progressThreshold{\min\{r^3 \LipGrad^{3/2} ( \| y \|_{1} + 1)^{3/2}, r^2 \LipGrad ( \| y \|_{1} + 1)\} }

This section features our main result, Theorem~\ref{thmMainResultNonconvex} which bounds the number of iterations that Algorithm~\ref{algIPM} uses to find an approximate Fritz John point by $\bigO{\mu^{-7/4}}$. At a high level this proof is similar to typical cubic regularization arguments \cite{nesterov2006cubic}: we argue that if the termination conditions are not satisfied at the \emph{next} iterate then we have reduced the log barrier function by at least $\Omega(\mu^{7/4})$. Before proving Theorem~\ref{thmMainResultNonconvex}, we prove the auxiliary Lemmas~\ref{lemNonconvexBarrierProgressBound} and \ref{lemNonconvexTerminateFirstOrder}. Lemma~\ref{lemNonconvexBarrierProgressBound} shows we reduce the barrier merit function when the predicted progress at each iteration is large; Lemma~\ref{lemNonconvexTerminateFirstOrder} allows us to reason about when the algorithm will terminate.

Also recall that $\tau_{l}$, $\tau_{c}$ and $\mu$ are all parameters for our termination criteria \eqref{eq:first-order-SIP-full:first-order}. To simplify the analysis we assume $\mu$ is small enough such that Assumption~\ref{assumption:nonconvex-pars} holds. Assumption~\ref{assumption:nonconvex-pars} also fixes the value of $\tau_{c}$ relative to other parameters.
Assumption~\ref{assumption:nonconvex-pars} can be readily relaxed (see Remark~\ref{remark:remove-assumption-nonconvex-pars}).

\newcommand{\valCentering}{\frac{\tau_{l}^2 \mu}{\LipGrad}}
\newcommand{\valNCLipRatio}{\frac{\LipHess^2 \mu}{\LipGrad^{3}}}

\begin{assumption}[Sufficiently small $\mu$]\label{assumption:nonconvex-pars}
Let
\begin{flalign*}
%%% first equation %%%%
%\mytag{\assumeNC.$\tau_{c}$}
\tau_{c}= \left(\valCentering\right)^{1/2} &\in(0,1] \\
%%%%%% second equation %%%%%%%%
%\mytag{\assumeNC.$\mu$}
\valNCLipRatio &\in ( 0,  1].
\end{flalign*}
\end{assumption}

Lemma~\ref{lemNonconvexBarrierProgressBound} provides a bound on the progress as a function of the parameter $\eta \in [0,1]$ which controls the step size. This allows us to guarantee that during Algorithm~\ref{algIPM} if the predicted progress from solving the trust-region subproblem $\barrierModel_{x} (\dir{x})$ is sufficiently large then we reduce the barrier function.
The proof of Lemma~\ref{lemNonconvexBarrierProgressBound} consists of two parts. The first part uses \eqref{eq:nonconvex-s-bound}, and the definition of $\alpha$ to argue that $\barrierModel_{x} (\alpha \dir{x}) \le \max\{\barrierModel_{x} (\dir{x}) , -\eta^2 \mu/ 3\}$. The second part uses  Lemma~\ref{lemPrimalDualApproxResultTwo} to show that $\barrierModel_{x} (\alpha \dir{x})$ accurately predicts the reduction in the barrier function.

\def\etaSIntervalOne{[0,1/5]} % change this!
\begin{restatable}{lemma}{lemNonconvexBarrierProgressBound}\label{lemNonconvexBarrierProgressBound}
Suppose Assumptions~\ref{assume-lip-deriv} and \ref{assumption:nonconvex-pars} hold (Lipschitz derivatives, and sufficiently small $\mu$). Let $x \in \X$, $\edit{\eta} \in \etaSIntervalOne$, 
\edit{$(x^{+}, y^{+}) \gets \callStepIPM{f,a,\mu, x, \eta}$}. Then \editThree{$\Convex{\{x,\xPlus\}} \subseteq \X$} and
\begin{flalign}\label{nonconvex:eq:lemNonconvexBarrierProgressBound:main}
\barrier(\xPlus) - \barrier(x) \le  \editThree{\frac{7}{3}} \mu \edit{\eta}^3 + \max \left\{ \barrierModel_{x} (\dir{x}), -\frac{\edit{\eta}^2 \mu}{3} \right\}.
\end{flalign}
\end{restatable}

\CHECKED

\begin{myproof}
	Our first goal is to show for all $\alpha \in (0,1]$ that
	\begin{flalign}\label{nonconvex:eq:predicted-progress}
		\barrierModel_{x} (\alpha \dir{x}) \le \max\left\{\barrierModel_{x} (\dir{x}) , -\frac{\edit{\eta}^2 \mu}{3}\right\}.
	\end{flalign}
	Note \eqref{nonconvex:eq:predicted-progress} trivially holds if $\alpha = 1$.
	Therefore let us consider the case $\alpha \in (0,1)$. In this case, 
	\begin{flalign}\label{nonconvex:eq:alpha-lb}
		\alpha &= \frac{\edit{\eta}}{\| S^{-1} \dir{s} \|_2} \ge \edit{\eta} \sqrt{\frac{\mu}{\LipGrad (\| y \|_1 + 1) \| \dir{x} \|_2^2 - 2 \barrierModel_{x} (\dir{x})}} \ge \edit{\eta} \sqrt{\frac{\mu}{\edit{\eta}^2 \mu / 4 - 2 \barrierModel_{x} (\dir{x})}} 
	\end{flalign}
	where the first inequality uses \eqref{eq:nonconvex-s-bound}, and the second inequality uses $\| \dir{x} \|_2 \le \editThree{r=} \frac{\edit{\eta}}{2} \sqrt{\frac{\mu}{\LipGrad( \| y \|_1 + 1)}}$. Furthermore, if $ \barrierModel_{x} (\dir{x}) \in \left[ -\frac{\edit{\eta}^2 \mu}{4}, 0 \right]$ from \eqref{nonconvex:eq:alpha-lb} we get $\alpha \ge \sqrt{4/3} > 1$; by contradiction we conclude $\barrierModel_{x} (\dir{x}) \not\in \left[ -\frac{\edit{\eta}^2 \mu}{4}, 0 \right]$. Using $\barrierModel_{x} (\dir{x}) \not\in \left[ -\frac{\edit{\eta}^2 \mu}{4}, 0 \right]$ and $\barrierModel_{x} (\dir{x}) \le \barrierModel_{x} (\zeros) = 0$ (recall definition of $\dir{x}$ in \callStepIPM{}), we deduce $\barrierModel_{x} (\dir{x})  < -\frac{\edit{\eta}^2 \mu}{4}$. Combining $\barrierModel_{x} (\dir{x})  < -\frac{\edit{\eta}^2 \mu}{4}$ with \eqref{nonconvex:eq:alpha-lb} yields $\alpha \ge \edit{\eta} \sqrt{\frac{\mu}{-3 \barrierModel_{x} (\dir{x})}}$. Therefore,
	\begin{flalign*}
		\barrierModel_{x} (\alpha \dir{x}) =  \alpha^2 \frac{1}{2}  \dir{x}^T \grad^2 \barrier(x) \dir{x} + \alpha \grad \barrier(x)^T \dir{x} \le \alpha^2 \barrierModel_{x} (\dir{x}) \le -\frac{\edit{\eta}^2 \mu}{3}
	\end{flalign*}
	where the first inequality follows by $\grad \barrier(x)^T \dir{x} \le 0$ as implied by \eqref{eq:trust-region-delta-gain} with \editThree{$g = \grad \barrier(x)$, $H = \grad^2 \barrier(x)$ and $u^{*} = \dir{x}$}, and the second by $\alpha \ge \edit{\eta} \sqrt{\frac{\mu}{-3\barrierModel_{x} (\dir{x})}}$. Thus \eqref{nonconvex:eq:predicted-progress} holds.
	
	It remains to bound the accuracy of the predicted decrease $\barrierModel_{x} (\alpha \dir{x})$. Note that by $\alpha \in [0,1]$, $\| \dir{x} \|_2 \le r$ we have
	\begin{flalign}\label{eq:alph-dir-x-bound}
		\| \alpha \dir{x} \|_2 \le \| \dir{x} \|_2 \le \editThree{r=} \frac{\edit{\eta}}{2} \sqrt{ \frac{\mu }{\LipGrad (\| y\|_1 + 1)}}.
	\end{flalign}
	Let us select $\kappa = (21/20) \edit{\eta}$, this choice satisfies the premise of Lemma~\ref{lemPrimalDualApproxResultTwo} because
	\begin{flalign}\label{nonconvex:eq:internal-kappa-bound}
		\alpha \| S^{-1} d_{s} \|_{2} + \frac{\LipGrad \| \alpha \dir{x} \|_2^2  \| y \|_2}{\mu}  \le \edit{\eta} + \frac{\edit{\eta}^2}{4} \le (21/20) \edit{\eta} = \kappa
	\end{flalign}
	where the first inequality comes from \editThree{$\alpha \le \eta / \| S^{-1} \dir{s} \|_2$ by Line~\ref{line:step-size-computation} of Algorithm~\ref{algIPM}} and \eqref{eq:alph-dir-x-bound}, and the third inequality uses $\edit{\eta} \in \etaSIntervalOne$. Since $\edit{\eta} \in \etaSIntervalOne$ we deduce $\kappa \le 1/4$ so the conditions of Lemma~\ref{lemPrimalDualApproxResultTwo} hold. Therefore, Lemma~\ref{lemPrimalDualApproxResultTwo} implies \editThree{$\Convex\{ x, \xPlus \} \subseteq \X$}, and
	\begin{flalign}
		\left| \barrier(x) + \barrierModel_{x} (\alpha \dir{x}) -  \barrier(\xPlus) \right| &\le \frac{\LipHess}{6} \left(1 + 2 \| y \|_1 \right) \| \alpha \dir{x} \|_2^3 + \editThree{\frac{4}{3}} \LipGrad \| \alpha \dir{x} \|_2^2 \| y \|_1 \kappa  + \editThree{\frac{5}{3}} \mu \kappa^3 \notag \\
		%%%%%%%
		&\le \frac{\LipGrad^{3/2} \mu^{-1/2}}{6}  \left(1 + 2 \| y \|_1 \right) \| \alpha \dir{x} \|_2^3  + \editThree{\frac{4}{3}} \LipGrad \| \alpha \dir{x} \|_2^2 \| y \|_1 \kappa + \editThree{\frac{5}{3}} \mu \kappa^3 \notag \\
		%%%%%%%
		&\le\left(\frac{2}{6 \times 2^3} + \editThree{\frac{4}{3}} (1/2^2) (21/20) + \editThree{\frac{5}{3}} (21/20)^3 \right)  \mu \edit{\eta}^3 \notag \\
		& \le \editThree{\frac{7}{3}} \mu \edit{\eta}^3 \label{nonconvex:eq:barrier-predicted-error-bound-234}
	\end{flalign}
	where the second inequality uses $\valNCLipRatio \in (0,1]$ from Assumption~\ref{assumption:nonconvex-pars}, the third inequality uses our bound on $\| \alpha \dir{x} \|_2$ and $\kappa$, i.e., \eqref{eq:alph-dir-x-bound} and \eqref{nonconvex:eq:internal-kappa-bound}.
	Combining \eqref{nonconvex:eq:predicted-progress} and \eqref{nonconvex:eq:barrier-predicted-error-bound-234} gives \eqref{nonconvex:eq:lemNonconvexBarrierProgressBound:main}.
\end{myproof}

\def\etaConstant{\editThree{50}}

\def\etaXintervalTFO{(0, \frac{1}{\edit{\etaConstant}} (\valCentering)^{1/4}]}
\newcommand{\minProgressPerIt}[0]{-\editThree{5 \eta^3 \mu}}

Lemma~\ref{lemNonconvexTerminateFirstOrder} shows that for Algorithm~\ref{algIPM} if the predicted progress, $\barrierModel_{x} (\dir{x})$, from the trust-region step is small then \eqref{eq:first-order-SIP-full:first-order} holds at $(\xPlus, \yPlus)$.
\editThree{Moreover, if the predicted progress from $\xPlus$ is small then 
\eqref{eq:first-order-SIP-full:second-order} also holds}. The proof of Lemma~\ref{lemNonconvexTerminateFirstOrder}  first uses \eqref{eq:nonconvex-s-bound} and $\barrierModel_{x} (\dir{x}) \ge \minProgressPerIt$ to argue that $\| S^{-1} \dir{s} \|_2$ and $\| Y^{-1} \dir{y} \|_2$ must be small. This enables the use of Lemma~\ref{lem:LagError} to bound $\| \grad \Lag (\xPlus, \yPlus) \|_2$ \editThree{and thereby showing \eqref{eq:first-order-SIP-full:first-order} holds}. \editThree{To derive the second-order guarrantees the proof lower bounds the minimum eigenvalue of $\grad^2 \psi(\xPlus)$ and then translates this into \eqref{eq:first-order-SIP-full:second-order} using that $\yPlus \approx \mu S^{-1} \ones$.}

\begin{restatable}{lemma}{lemNonconvexTerminateFirstOrder}\label{lemNonconvexTerminateFirstOrder}
\editThree{Let $\dir{x}$ and $\dir{\xPlus}$ correspond to the directions computed by Algorithm~\ref{algIPM} at the iterate 
$x$ and $\xPlus$ respectively.}
Suppose Assumptions~\ref{assume-lip-deriv} and \ref{assumption:nonconvex-pars} hold (direction selection, Lipschitz derivatives, and sufficiently small $\mu$). Further assume $x \in \X$, $\edit{\eta} \in \etaXintervalTFO$, and $\barrierModel_{x} (\dir{x}) \ge \minProgressPerIt$. 
\editThree{Under these assumptions, 
\edit{$(\xPlus, \yPlus) \gets \callStepIPM{f,a,\mu, x, \eta}$} satisfies \eqref{eq:first-order-SIP-full:first-order}. 
Additionally, if $\barrierModel_{\xPlus} (\dir{\xPlus}) \ge \minProgressPerIt$ then $(\xPlus,\yPlus)$ satisfies \eqref{eq:first-order-SIP-full:second-order}.}
%,and \editThree{$(x,y)$ with $S = \diag(a(x))$ and $y = \mu S^{-1}$ satisfies \eqref{eq:first-order-SIP-full:second-order}}.
\end{restatable}

\CHECKED

\begin{myproof}
	First, let us bound $\| S^{-1} \dir{s} \|_2$:
	\begin{flalign*}
		\| S^{-1} \dir{s} \|_2 &\le \sqrt{\frac{\LipGrad ( \| y \|_{1} + 1)  \| \dir{x} \|_2^2 - 2 \barrierModel_{x}(\dir{x})}{\mu}} \\
		&\le  \sqrt{ \edit{\frac{\eta^2}{4}} +  \editThree{\frac{3 \eta^2}{4}} } \editThree{= \eta}
	\end{flalign*}
	where the first inequality uses \eqref{eq:nonconvex-s-bound} and the second inequality uses $\| d_x \|_2 \le r = \edit{\frac{\eta}{2}} \sqrt{ \frac{\mu }{\LipGrad (\| y\|_1 + 1)}}$ and $\barrierModel_{x}(\dir{x}) \ge \editThree{\minProgressPerIt \ge_\star -\frac{3}{8}\eta^2 \mu}$ \editThree{where $\star$ uses that $\edit{\eta} \le \frac{1}{\etaConstant} \left( \valCentering\right)^{1/4} \in (0,1/\etaConstant]$}. \editThree{By Line~\ref{line:step-size-computation} of Algorithm~\ref{algIPM} it follows that $\alpha = 1$ and therefore $\xPlus = x + \dir{x}$ and $\yPlus = y + \dir{y}$. Moreover, by Lemma~\ref{lemNonconvexBarrierProgressBound} we have $\Convex\{\x,\xPlus\} \subseteq \X$.}
	
	Furthermore, by Lemma~\ref{lem:complementary}, the fact \editThree{that} $y = \mu S^{-1} \ones$, and our bound on $\| S^{-1} \dir{s} \|_2$ we have
	\begin{flalign}\label{eq:Yinv-dir-y}
		\| Y^{-1} \dir{y} \|_2 \le \| S^{-1} \dir{s} \|_2 \le \editThree{\eta}
	\end{flalign}
	\editThree{Let $S^{+} := \diag(\cons(\xPlus))$. Also from Lemma~\ref{lem:complementary} we get}
	\begin{flalign}
	\notag	\| S^{+} y^{+} - \mu \ones \|_{2} &\le \mu \| S^{-1} \dir{s} \|_{2} \| Y^{-1} \dir{y} \|_{2} + \frac{\LipGrad}{2}  \| y \|_{2} (1 + \| Y^{-1} d_{y} \|_{2}) \| \dir{x} \|_2^2 \\
		&\le \editThree{\mu \eta^2} +  \edit{\frac{\mu \eta^2}{4}} \editThree{= \frac{5}{4} \mu \eta^2} \le  \frac{\mu}{\editThree{2000}} \left( \valCentering\right)^{1/2} = \frac{\mu \tau_{c}}{\editThree{2000}},
		\label{eq:strong-comp-bound-L2-norm}
	\end{flalign}
	where the second inequality uses $\| Y^{-1} d_{y} \|_{2} \le \| S^{-1} d_{s} \|_{2}  \le \editThree{\eta} \le 1$ and $\| \dir{x} \|_2 \le r = \edit{\frac{\eta}{2}} \sqrt{\frac{\mu}{\LipGrad (\| y \|_1 + 1)}}$, and the third inequality $\edit{\eta} \in \etaXintervalTFO$.
	\editThree{Inequality \eqref{eq:strong-comp-bound-L2-norm} establishes \eqref{eq:first-order-SIP-full:dual-gap}.
	By \eqref{eq:Yinv-dir-y}, $\eta \le 1$, \eqref{eq:strong-comp-bound-L2-norm} and Lemma~\ref{lem:complementary} we get \eqref{eq:first-order-SIP-full:feasible}.

	The next step in the proof is to establish \eqref{eq:first-order-SIP-full:grad-lag} by bounding the terms $\| \delta \dir{x} - \grad_{x} \Lag( x^{+}, y^{+}) \|_2$ and $\| \delta \dir{x} \|_2$.
	First, we bound $\| \delta \dir{x} - \grad_{x} \Lag( x^{+}, y^{+}) \|_2$:}
	\begin{flalign*}
		\| \delta \dir{x} - \grad_{x} \Lag( x^{+}, y^{+}) \|_2 &\le  \LipGrad \| \dir{x} \|_2 \| y \|_2 \| Y^{-1} \dir{y} \|_2 + \frac{\LipHess}{2}  (\| y \|_1 + 1) \| \dir{x} \|^2_2 \\
		&\le \LipGrad \| \dir{x} \|_2 \| y \|_2 \| Y^{-1} \dir{y} \|_2 + \frac{ \LipGrad^{3/2} \mu^{-1/2}}{2}  (\| y \|_1 + 1) \| \dir{x} \|^2_2 \\
		&\le \editThree{\LipGrad \frac{\eta}{2} \sqrt{\frac{\mu}{\LipGrad( \| y \|_1 + 1)}} \| y \|_2 \eta + \frac{ \LipGrad^{3/2} \mu^{-1/2}}{2}  (\| y \|_1 + 1) \left( \frac{\eta}{2} \sqrt{\frac{\mu}{\LipGrad( \| y \|_1 + 1)}} \right)^2 }  \\
		&\editThree{= \frac{\eta^2 \sqrt{\mu \LipGrad} }{2} 
		\left( \frac{\| y \|_2}{\sqrt{\| y \|_1 + 1}}  + \frac{1}{4} \right)} \\
		&\le  \editThree{\frac{\eta^2 \sqrt{\mu \LipGrad} }{2} \sqrt{\| y \|_1 + 1}} \\
		&\le  \editThree{\frac{\tau_{l} \mu}{5000} \sqrt{\| y \|_1 + 1}}
		%\frac{\mu \sqrt{\| y \|_1 + 1}}{2}
	\end{flalign*}
	where the first inequality follows from Lemma~\ref{lem:LagError}, the second by $\valNCLipRatio \in (0,1]$, the third inequality using the bound $\| Y^{-1} \dir{y} \|_{2} \le \editThree{\eta}$ that \editThree{\eqref{eq:Yinv-dir-y} established} and $\| \dir{x} \|_2 \le \edit{\frac{\eta}{2}} \sqrt{\frac{\mu}{\LipGrad( \| y \|_1 + 1)}}$, the fourth inequality \editThree{uses $\| y \|_2 \le \| y \|_1$} and the final inequality uses $\eta \in \etaXintervalTFO$.
	
	Next, we bound  $\delta \| \dir{x} \|_2$.
	\edit{Using $S \dir{y} + Y \dir{s} + S y = \mu \ones$ and substituting $\dir{s} = \grad \cons(x) \dir{x}$ into $(\grad^2 \barrier ( x ) + \eye \delta) \dir{x} = - \grad \barrier ( x )$ and \edit{$\grad^2 \barrier(x) = \grad_{xx}^2 \Lag(x,y)$} we deduce that }
	$\grad_{x} \Lag( x , y ) + \grad_{xx}^2 \Lag (x, y )^T \dir{x} - \dir{y}^T  \grad_{x} \cons(x) = \delta \dir{x}$.
	Moreover, 
	\[ 
	\delta \| \dir{x} \|_2 \le \delta r \editThree{\le_{(a)} \frac{10 \eta^3 \mu}{r} =_{(b)} 20 \eta^2 \sqrt{\LipGrad( \| y \|_1 + 1) \mu}   \le_{(c)} \frac{20}{2500} \left(\valCentering\right)^{1/2} \sqrt{\LipGrad( \| y \|_1 + 1) \mu} = \frac{\tau_{l} \mu \sqrt{\| y \|_1 + 1}}{125}}
	\]
	\editThree{where $(a)$ uses $\minProgressPerIt \le \barrierModel_{x}(u) \le -\frac{\delta r^2}{\editThree{2}}$ by \eqref{eq:delta-bound:lem:trust-region-facts}, $(b)$ uses $r = \edit{\frac{\eta}{2}} \sqrt{\frac{\mu}{\LipGrad( \| y \|_1 + 1)}}$
	and $(c)$ uses $\eta \in \etaXintervalTFO$}. 
	Therefore using the bounds on $\| \delta \dir{x} - \grad_{x} \Lag( x^{+}, y^{+}) \|_2$ and $\delta \| \dir{x} \|_2$ that we proved,
	\begin{flalign*}
		\| \grad_{x} \Lag( x^{+}, y^{+}) \|_2 &\le \| \delta \dir{x} - \grad_{x} \Lag( x^{+}, y^{+}) \|_2 + \delta \| \dir{x} \|_2 \\
		&\le \editThree{ \frac{\tau_{l} \mu}{125} \sqrt{1 + \| y \|_1} + \frac{ \tau_{l} \mu}{5000} \sqrt{1 + \| y \|_1}} \\
		&\le \tau_{l} \mu \sqrt{1 + \| y \|_1}.
	\end{flalign*}
	Therefore \eqref{eq:first-order-SIP-full:first-order} holds.

\newcommand{\rPlus}{r^{+}}
\editThree{Finally, we prove \eqref{eq:first-order-SIP-full:second-order} when $\barrierModel_{\xPlus} (\dir{\xPlus}) \ge \minProgressPerIt$ also holds.
Let $v_{\min}$ be the eigenvector of $\grad^2 \barrier(\xPlus)$ corresponding to the minimum eigenvalue of $\grad^2 \barrier(\xPlus)$, and $\rPlus$ be the radius choosen to compute $\dir{\xPlus}$.
Then we have 
\begin{flalign}\label{eq:second-order-predicted-progress-relationship}
\minProgressPerIt \le \barrierModel_{\xPlus} (\dir{\xPlus}) \le \min\{ \barrierModel_{\xPlus} (\rPlus v_{\min} ),  \barrierModel_{\xPlus} (- \rPlus  v_{\min}) \}  \le \frac{\lambda_{\min}(\grad^2 \barrier(\xPlus) ) (\rPlus)^2}{2}
\end{flalign}
where $\lambda_{\min}(\cdot)$ denotes the minimum eigenvalue. Therefore, with $S^{+} = \diag(\cons(\xPlus))$ we have
\begin{flalign}
\notag
\lambda_{\min}(\grad^2 \barrier(\xPlus)) &\ge_{(a)} \minProgressPerIt \times \frac{2}{(\rPlus)^2} 
=_{(b)} -5 \eta^3 \mu  \times \frac{8 \LipGrad (1 + \| \mu (S^{+})^{-1} \ones \|_1)}{\eta^2 \mu} \\
%%%%%%%%%%%%%%%%%%%%%%%%%%%%%%%%%%%%%%%%%%%%%%%%%%%%%%%%%%%%%%%%%%%%%%
\notag
&= -40 \eta \LipGrad (1 + \| \mu (S^{+})^{-1} \ones \|_1) \ge_{(c)} -\frac{40}{50} \left(\valCentering\right)^{1/4} \LipGrad ( 1 + \| \mu (S^{+})^{-1} \ones \|_1)  \\
%%%%%%%%%%%%%%%%%%%%%%%%%%%%%%%%%%%%%%%%%%%%%%%%%%%%%%%%%%%%%%%%%%%%%%
&= -\frac{4}{5} \tau_{c}^{1/2} \LipGrad ( 1 + \| \mu (S^{+})^{-1} \ones \|_1)
\label{eq:min-eig-bound-with-S-plus}
\end{flalign}
where $(a)$ rearranges \eqref{eq:second-order-predicted-progress-relationship}, $(b)$ uses $\rPlus = \frac{\eta}{2} \sqrt{\frac{\mu}{\LipGrad (\| \mu (S^{+})^{-1} \ones \|_1 + 1)}}$, and $(c)$ uses $\eta \in \etaXintervalTFO$.
Next, we have 
\begin{flalign}
\notag &\| \grad_{xx}^2 \Lag(\xPlus, \yPlus)	+ \mu \grad a(\xPlus)^T (S^{+})^{-2} \grad a(\xPlus) - \grad^2 \barrier(\xPlus) \|_2 \\
%%%%%%%%%%%%%%%%%%%%%%%%%%%%%%%%%%%%%
\notag
&= \| \grad_{xx}^2 \Lag(\xPlus, \yPlus) -  \grad_{xx}^2 \Lag(\xPlus, \mu (S^{+})^{-1} \ones) \|_2 
= \left\| \sum_{i \in \ConSet} \left( \frac{\mu}{\cons_i(\xPlus)} - y_i \right) \grad^2 \cons_i(x) \right\|_2 \\
%%%%%%%%%%%%%%%%%%%%%%%%%%%%%%%%%%%%%
\notag 
&\le_{(a)} \LipGrad \sum_{i \in \ConSet} \abs{ \yPlus_i -  \frac{\mu}{\cons_i(\xPlus)} } = \LipGrad \sum_{i \in \ConSet} \yPlus_i \abs{1 - \frac{\mu}{\cons_i(\xPlus) \yPlus_i}} \le_{(b)} \LipGrad \| \yPlus \|_1  \max\left\{ 1 - \frac{\mu}{\mu + \frac{\mu \tau_{c}}{2000}}, \frac{\mu}{\mu - \frac{\tau_{c} \mu}{2000}} - 1 \right\} \\
%%%%%%%%%%%%%%%%%%%%%%%%%%%%%%%%%%%%%
&= \LipGrad \| \yPlus \|_1 \max\left\{ \frac{\frac{\tau_{c}}{2000}}{1 + \frac{\tau_{c}}{2000}}, \frac{\frac{\tau_{c}}{2000}}{1 - \frac{\tau_{c}}{2000}} \right\} \le_{(c)} \frac{\tau_{c} \LipGrad \| \yPlus \|_1}{1999} 
\label{eq:Hessian-barrier-minus-Lagrangian}
\end{flalign}
where $(a)$ uses that $\grad a_i(x)$ is $\LipGrad$-Lipschitz, $(b)$ uses \eqref{eq:strong-comp-bound-L2-norm}, and $(c)$ uses $\tau_{c} \in (0,1]$.
Furthermore,
\begin{flalign}\label{eq:bound-mu-SY-plus-inverse}
\| \mu (S^{+} Y^{+})^{-1} \ones \|_{\infty} \le \mu \max_{i \in \ConSet} (\cons_i(\xPlus) y^{+}_i)^{-1} \le \frac{\mu}{\min_{i \in \ConSet} \cons_i(\xPlus) y^{+}_i} \le_{(a)} \frac{\mu}{\mu - \frac{\tau_{c} \mu}{2000}} \le_{(b)} \frac{2000}{1999}
\end{flalign} 
where $(a)$ uses \eqref{eq:strong-comp-bound-L2-norm} and $(b)$ uses $\tau_{c} \le 1$.
Therefore,
\begin{flalign}
\| \mu (S^{+})^{-1} \ones \|_1 = \mu Y^{+} \ones \cdot (S^{+} Y^{+})^{-1} \ones \le \| \yPlus \|_1 \| \mu (S^{+} Y^{+})^{-1} \ones \|_{\infty} \le \frac{2000}{1999} \| \yPlus \|_1
\label{eq:bound-S-plus-inverse-norm-by-y-plus-norm}
\end{flalign}
where the first inequality uses H\"{o}lder's inequality and the second inequality uses \eqref{eq:bound-mu-SY-plus-inverse}.
Finally,  
\begin{flalign*}
\grad_{xx}^2 \Lag(\xPlus, \yPlus) + \mu \grad a(\xPlus)^T (S^{+})^{-2} \grad a(\xPlus) &\succeq 
\grad^2 \barrier(\xPlus) - \frac{\tau_{c}}{1999} \LipGrad \| \yPlus \|_1 \eye \\
&\succeq 
-\frac{4}{5} \tau_{c}^{1/2} \LipGrad ( 1 + \| \mu (S^{+})^{-1} \ones \|_1) \eye - \frac{\tau_{c}}{1999} \LipGrad \| \yPlus \|_1 \eye \\[5pt]
&\succeq 
-\frac{5}{6} \tau_{c}^{1/2} \LipGrad ( 1 + \| \yPlus \|_1) \eye - \frac{\tau_{c}}{1999} \LipGrad \| \yPlus \|_1 \eye \\
&\succeq 
-\tau_{c}^{1/2} \LipGrad ( 1 + \| \yPlus \|_1) \eye
\end{flalign*}
where the first transition uses \eqref{eq:Hessian-barrier-minus-Lagrangian}, the second transition uses 
\eqref{eq:min-eig-bound-with-S-plus}, the third transition uses \eqref{eq:bound-S-plus-inverse-norm-by-y-plus-norm} and the final transition uses $\tau_{c} \in (0,1]$.
Therefore \eqref{eq:first-order-SIP-full:second-order} holds at $(\xPlus, \yPlus)$ as desired.
}
\end{myproof}

With Lemma~\ref{lemNonconvexBarrierProgressBound} and \ref{lemNonconvexTerminateFirstOrder} in hand we are now ready to prove our main result, Theorem~\ref{thmMainResultNonconvex}. The idea of the proof is that if over two consecutive iterations the function is not reduced by $\Omega(\mu^{7/4})$ then \eqref{eq:first-order-SIP-full:first-order} and \eqref{eq:first-order-SIP-full:second-order} hold. This argument is a little different from proofs of related results in literature. Convergence proofs for cubic regularization \cite{nesterov2006cubic,cartis2011adaptive} argue that if there is a little progress this iteration then the next iterate will satisfy the termination criteria; convergence proofs for gradient descent argue that if there is little progress this iteration then the current iteration satisfies the termination criteria. The reason for our unusual argument is that Lemma~\ref{lemNonconvexTerminateFirstOrder} guarantees that the \editThree{termination criteria 
holds only if both the current and next iterate have small predicted progress}.

\begin{restatable}{theorem}{thmMainResultNonconvex}\label{thmMainResultNonconvex}
Suppose Assumptions~\ref{assume-lip-deriv}, \ref{assume:barrier-and-initial-point} and \ref{assumption:nonconvex-pars} hold (Lipschitz derivatives, \editThree{barrier function bounded below} and sufficiently small $\mu$). Then $\callAlgMain{\obj, \cons, \mu, \tau_{l}, \LipGrad, \edit{\eta}, \x\ind{0}}$ with
\begin{flalign}\label{eq:choose-eta-nonconvex}
\edit{\eta} = \frac{1}{\etaConstant} \left(\frac{\tau_{l}^2 \mu}{\LipGrad} \right)^{1/4},
\end{flalign}
takes at most
$$
\bigO{ 1 + \frac{\barrier(x\ind{0}) - \barrier^{*}}{\mu} \left(\frac{\LipGrad}{\mu \tau_{l}^2} \right)^{3/4} }
$$
iterations to terminate with a $(\mu,\tau_{l},\tau_{c})$-approximate second-order SIP $(x^{+}, y^{+})$, i.e., \eqref{eq:first-order-SIP-full:first-order} and \eqref{eq:first-order-SIP-full:second-order} hold.
\end{restatable}

\CHECKED

\begin{myproof}
	Let $x \in \X$ be some iterate of the algorithm with corresponding direction $\dir{x}$. If $\minProgressPerIt \ge \barrierModel_{x} (\dir{x})$ then
	\begin{flalign}
		\notag \barrier(x + \alpha \dir{x}) - \barrier(x)  &\le  \editThree{\frac{5}{3}} \mu \edit{\eta}^3 + \max \left\{ \barrierModel_{x} (\dir{x}), -\frac{\edit{\eta}^2 \mu}{3} \right\} \\
		%%%%%%%%%%%%%%%%%%%%%%%%
		&\le \editThree{ \frac{5}{3} \mu \edit{\eta}^3 + \max \left\{ \minProgressPerIt, -\frac{\edit{\eta}^2 \mu}{3} \right\} = \editThree{ \frac{5}{3} \mu \edit{\eta}^3 - 5 \mu \eta^3}
		= -\editThree{\frac{10}{3}} \mu \edit{\eta}^3} \label{nonconvex:eq:main-result-nonconvex:large-progress-bound}
	\end{flalign}
	where the first transition uses Lemma~\ref{lemNonconvexBarrierProgressBound}, the second transition uses $\barrierModel_{x} (\dir{x}) \ge \minProgressPerIt$, and the third transition uses \editThree{$\edit{\eta} \le 1/15$}.
	
	Let $(\x,\dir{\x}, \alpha)$ denote the current primal iterate, direction and step size.
	Let $(\xPlus,\dir{\xPlus}, \alpha^{+})$ denote the subsequent primal iterate, direction and step size.
	By Lemma~\ref{lemNonconvexTerminateFirstOrder} if $\minProgressPerIt \le \barrierModel_{x} (\dir{x})$ then \eqref{eq:first-order-SIP-full:first-order} holds at $(\xPlus, \yPlus)$. 
	Also, by Lemma~\ref{lemNonconvexTerminateFirstOrder} if $\minProgressPerIt \le \barrierModel_{\xPlus} (\dir{x}^{+})$ then 
	\eqref{eq:first-order-SIP-full:second-order} holds at $(\xPlus, \yPlus)$. Therefore if both $\minProgressPerIt \le \barrierModel_{x} (\dir{x})$  and $\minProgressPerIt \le \barrierModel_{\xPlus} (\dir{x}^{+})$ the algorithm terminates at at $(\xPlus, \yPlus)$. 
	
	It remains to show that if either $\barrierModel_{x} (\dir{x}) < \minProgressPerIt$ or $\barrierModel_{\xPlus} (\dir{x}^{+}) < \minProgressPerIt$ then over these two iterations we reduce the function value by a constant quantity. First note that even if $\barrierModel_{x} (\dir{x}) \ge \minProgressPerIt$ we by $\barrierModel_{x} (\dir{x}) \le 0$ we still have
	\begin{flalign}
		\barrier(x + \alpha \dir{x}) - \barrier(x)  &\le \editThree{\frac{7}{3}} \mu \edit{\eta}^3 + \max \left\{ \barrierModel_{x} (\dir{x}), -\frac{\edit{\eta}^2 \mu}{3} \right\} \le \editThree{\frac{7}{3}} \mu \edit{\eta}^3 \label{nonconvex:eq:main-result-nonconvex:no-progress-bound}
	\end{flalign}
	where the first inequality follows from Lemma~\ref{lemNonconvexBarrierProgressBound}. The same equation applies replacing $(\x,\dir{x}, \alpha)$ with $(\xPlus,\dir{\xPlus}, \alpha^{+})$. By applying \eqref{nonconvex:eq:main-result-nonconvex:large-progress-bound} and \eqref{nonconvex:eq:main-result-nonconvex:no-progress-bound} we can see that if over these two iterations the algorithm did not terminate then $\barrier$ must have been reduced by at least \editThree{$\frac{10}{3} \mu \eta^3 - \editThree{\frac{7}{3}} \mu \eta^3 \ge \mu \eta^3$}.
	To conclude note if the algorithm has not terminated across iterations $0, \dots, K$ then letting $x\ind{k}$ be the $k$th $x$ iterate,
	$\barrier(x\ind{0}) - \barrier^{*}  \ge \sum^{K-1}_{k=0} ( \barrier(x\ind{k}) - \barrier(x\ind{k+1}) ) \ge \frac{K-2}{2} \times \editThree{\mu \eta^3}$, rearranging to bound $K$ \editThree{ and substituting for $\eta$ using \eqref{eq:choose-eta-nonconvex}} gives the result.
\end{myproof}

\edit{
\begin{remark}\label{remark:remove-assumption-nonconvex-pars}
Assumption~\ref{assumption:nonconvex-pars} can be readily removed from Theorem~\ref{thmMainResultNonconvex}, for
example, suppose we wish to find a $(\mu,\tau_{l},\tau_{c})$-approximate second-order SIP which does not satisfy Assumption~\ref{assumption:nonconvex-pars} then we can set:
\begin{flalign*}
\mu' &= \min\left\{ \mu, \frac{\LipGrad^3}{\LipHess^2} \right\} \\
\tau_{l}' &= \min\left\{ \tau_{l}, \editThree{\tau_{c}} \sqrt{\frac{\LipGrad}{\mu'}}  \right\}  \\
\tau_{c}' &= \left( \frac{ (\tau_l')^2 \mu'}{\LipGrad} \right)^{1/2}.
\end{flalign*} 
Substituting these values into Theorem~\ref{thmMainResultNonconvex} gives an iteration bound of
\[
\bigO{ 1 + \left( \barrier(x\ind{0}) - \barrier^{*} \right) \left( \mu^{-7/4} \left(\frac{\LipGrad}{\tau_{l}^2} + \frac{\mu}{\editThree{\tau_{c}^2}} \right)^{3/4} + \left( \LipHess^2 / \LipGrad^3 \right)^{7/4} \left(\frac{\LipGrad}{\tau_{l}^2} + \frac{\LipHess^2}{\LipGrad^3 \editThree{\tau_{c}^2} } \right) \right) }.
\]
\end{remark}
}

\newcommand{\valCLipRatio}{\frac{\LipHess^3 \mu}{\LipGrad^{4} \tau_{l}}}

\edit{
\section{Comparison with existing results}\label{sec:comparison-with-existing-results}
\editThree{This section compares against other methods for constrained nonconvex optimization in the literature
in how their worst-case iteration bounds scale with termination tolerances.}
One difficulty with \editThree{nonconvex constrained optimization} is that there are many choices \editThree{of} termination criteria and this choice affects iteration bounds. 
\editThree{We focus on comparing with \citet{birgin2016evaluation}.}
\citet{birgin2016evaluation} guarantee to find an unscaled KKT points or a certificate of local infeasibility. Their criteria is different from our approximate Fritz John termination criteria. Therefore for the sake of comparison we now introduce a new pair of termination criteria similar to the criteria they presented. Our own definition of an unscaled KKT point is
\begin{subequations}\label{eq:KKT}
	\begin{flalign}
		\cons(x) &\ge -\epsOpt \ones  \\
		\| \grad_{x} \Lag(x,y) \|_{2} &\le \epsOpt \label{eq:KKT:dual-feas} \\
		y &\ge \zeros \\
		\cons_i(x) y_i &\le \epsOpt \quad \forall i \in \ConSet. \label{eq:KKT-comp}
	\end{flalign}
\end{subequations}
Let us contrast this definition with the definition of an unscaled KKT point given in \citet[Equation (2.8)]{birgin2016evaluation}. The most important difference is how complementarity is measured\footnote{\editThree{There are also differences in the norm used to measure feasibility (they use Euclidean norm we use infinity norm) but this difference is not significant as 
this section focuses on comparing methods in terms of their rate of convergence only with respect to the termination tolerances.}}. In particular, in \citet{birgin2016evaluation} their termination criteria replaces \eqref{eq:KKT-comp} of our criteria with $\min\{ \cons_i(x), y_i \} \le \epsOpt$. In this respect, the termination criteria of \citet{birgin2016evaluation} is stronger than \eqref{eq:KKT}. 
To detect infeasibility we consider the following termination criteria. 
\begin{subequations}\label{eq:Linf-infeas} 
	\begin{flalign}
		\min_{i \in \ConSet} {a_i(x)} &< -\epsOpt / 2 \label{eq:Linf-infeas:con-violation} \\
		\left\| \grad \cons(x)^T y  \right\|_2 &\le \epsInf \label{eq:Linf-infeas:dual-grad-cons} \\
		\| y \|_1 &= 1 \label{eq:Linf-infeas:y-L1-norm} \\
		a(x) + t \ones &\ge \zeros \label{eq:Linf-infeas:t-cons} \\
		(a_i(x) + t) y_i  &\le \epsInf \epsOpt \quad \forall i \in \ConSet \label{eq:Linf-infeas:comp} \\
		y &\ge \zeros \label{eq:Linf-infeas:y-nonnegative}
	\end{flalign}
\end{subequations}
System~\eqref{eq:Linf-infeas} finds an approximate KKT point for the problem of minimizing the infinity norm of the constraint violation \editThree{which has at least $\epsOpt / 2$ violation of constraints}. In contrast, \citet{birgin2016evaluation} detect infeasibility by finding a stationary point for the Euclidean norm of the constraint violation \textbf{squared} which they denote by $\theta(x)$. 
\editThree{In particular, using our notation, \editThree{$\theta(x) = \| \min\{\cons(x), \zeros \} \|^2$ and} they declare a point infeasible if 
$\theta(x) \ge 0.99 \epsOpt^2$ and $\| \grad \theta(x) \| \le \epsInf \epsOpt$ \cite[Equation~(2.14) with $\psi = \epsInf \epsOpt$]{birgin2016evaluation}.
Their infeasibility certificate is equivalent to finding a solution to the following system with $z_i = \max\{ -\cons_i(x), 0 \}$, $y = \frac{z}{ \| z \|_2}$}
\begin{subequations}\label{eq:L2-infeas} 
	\begin{flalign}
		& \editThree{\| z \|_2 \ge 0.99 \epsOpt} \label{eq:L2-infeas:con-violation} \\
		& \left\| \grad \cons(x)^T y  \right\|_2 \le \editThree{\frac{\epsInf \epsOpt}{\| z \|_2}} \label{infeas-constraint} \\
		& y = \frac{z}{\| z \|_2} \label{eq:L2-infeas:y-equals-z-direction} \\
		& a(x) + z \ge \zeros \\
		& (a_i(x) + z_i) y_i  = 0  \quad \forall i \in \ConSet \label{eq:L2-infeas:comp} \\
		& \editThree{z}, y \ge \zeros. 
	\end{flalign}
\end{subequations}
\editThree{Note that in the challenging case for declaring infeasibility, i.e., $\theta(x) = 0.99 \epsOpt^2$
then \eqref{infeas-constraint} becomes $\left\| \grad \cons(x)^T y  \right\|_2 \le \frac{\epsInf}{0.99}$, in which case 
\eqref{eq:Linf-infeas} and \eqref{eq:L2-infeas} become similar \editThree{(recall
\eqref{eq:Linf-infeas} is minimizing the $\ell_{\infty}$-norm of constraint violation and \eqref{eq:L2-infeas} is minimizing 
the Euclidean norm of constraint violation)}.}
\editThree{
Moreover, for $\epsInf \in (0,1/(4 \NumCon)]$ if either \eqref{eq:Linf-infeas} and \eqref{eq:L2-infeas} are satisfied then:
\begin{flalign}\label{eq:weak-farkas}
-\cons(x)^T y \ge \frac{\epsOpt}{4}, \quad \left\| \grad \cons(x)^T y  \right\|_2 \le \epsInf, \quad y \ge \zeros
\end{flalign}
holds.
This is an approximate Farkas certificate of primal infeasibility \cite{andersen2001certificates}, generalized to nonlinear constraints. 
By \cite[Observation 1]{hinder2018one}, \eqref{eq:weak-farkas} proves infeasibility \editThree{inside an $\ell_{\infty}$}-ball of radius $R$ if \editThree{$\epsInf \le  \frac{\epsOpt}{2 R \times 4 \sqrt{\NumCon}}$}, $f$ is convex, and $a_i$ is concave. 
We now derive \eqref{eq:weak-farkas}.
If \eqref{eq:Linf-infeas} holds then
\begin{flalign*}
-\cons(x)^T y &>_{(a)} t \| y \|_1 - \NumCon \epsOpt \epsInf =_{(b)} t - \NumCon \epsOpt \epsInf \ge_{(c)} - \min_{i \in \ConSet} \cons_i(x) - \NumCon \epsOpt \epsInf >_{(d)} \frac{\epsOpt}{2} - \NumCon \epsOpt \epsInf \\
&>_{(e)} \frac{\epsOpt}{4}
\end{flalign*}
where $(a)$ uses \eqref{eq:Linf-infeas:comp} and \eqref{eq:Linf-infeas:y-nonnegative}, $(b)$ uses \eqref{eq:Linf-infeas:y-L1-norm}, $(c)$ uses \eqref{eq:Linf-infeas:t-cons}, $(d)$ uses \eqref{eq:Linf-infeas:con-violation} and $(e)$ uses $\epsInf \in (0,1/(4 \NumCon)]$.
Similarly, if \eqref{eq:L2-infeas} holds then by \eqref{eq:L2-infeas:comp}, \eqref{eq:L2-infeas:y-equals-z-direction} and \eqref{eq:L2-infeas:con-violation} respectively we have
\begin{flalign*}
-\cons(x)^T y = z^T y = \| z \|_2 \ge 0.99 \epsOpt.
\end{flalign*}
}
To obtain our algorithm that finds a point satisfying either \eqref{eq:KKT} or \eqref{eq:Linf-infeas}, we apply $\callAlgMain{}$ in two-phases (see $\callTwoPhase{}$ in Appendix~\ref{sec:algTwoPhase}). 

Let $x\ind{0} \in \reals^{\NumVar}$ be our starting point and define 
\begin{flalign*}
	t\ind{0} \defeq  \frac{\epsOpt}{2} + \max\{ \min_{i \in \ConSet} -\cons_i(x\ind{0}), 0 \}.
\end{flalign*}
Phase-one applies Algorithm~\ref{algIPM} to minimize the infinity norm of the constraint violation, i.e., we find an approximate Fritz John point of
\begin{subequations}\label{eq:phase-one-new}
	\begin{flalign}
		& \min_{x,t}{\obj\ind{P1}(x,t) := t} \\
		&\cons\ind{P1}(x,t) := \begin{pmatrix} \cons(x) + t \ones \\
			t  \\
			\frac{\epsOpt}{2} + t\ind{0} - t
		\end{pmatrix} \ge \zeros.
	\end{flalign}
\end{subequations}
Let $(x\ind{P1},t\ind{P1})$ be the solution obtained. Starting from $x\ind{P1}$, phase-two minimizes the objective subject to the ($\epsOpt$-relaxed) constraints, i.e., we find an approximate Fritz John point of
\begin{subequations}\label{eq:phase-two-new}
	\begin{flalign}
		& \min_{x}{\obj(x)} \\
		& \cons\ind{P2}(x) := \cons(x) + \epsOpt \ones \ge \zeros
	\end{flalign}
\end{subequations}
starting from the point obtained in phase-one.

We replace Assumption~\ref{assume-lip-deriv} with Assumption~\ref{assume-lip-deriv-two-phase}, where $\X$ is replaced with two sets, corresponding to phase-one and phase-two respectively:
\begin{flalign*}
	\tilde{\X}\ind{P1}  &\defeq \{ x \in \reals^{\NumVar} : \cons( x )  \ge - (\epsOpt/2 + t\ind{0})  \ones \} \\
	\tilde{\X}\ind{P2}  &\defeq \{ x \in \reals^{\NumVar} : \cons( x )  \ge -\epsOpt  \ones \}.
\end{flalign*}
By the definition of $t\ind{0}$ we have $\tilde{\X}\ind{P2} \subseteq \tilde{\X}\ind{P1}$.

\begin{assumption}\label{assume-lip-deriv-two-phase}
	Assume that  each $a_i : \reals^{\NumVar} \rightarrow \reals$ for $i \in \{1 ,\dots, \NumCon\}$ is a continuous function on $\reals^{\NumVar}$.
	Let $\LipGrad, \LipHess \in (0,\infty)$.
	The functions $\cons_i : \reals^{\NumVar} \rightarrow \reals$ have $\LipGrad$-Lipschitz first derivatives and $\LipHess$-Lipschitz second derivatives on the set $\tilde{\X}\ind{P1}$.
	The function $f : \reals^{\NumVar} \rightarrow \reals$ has $\LipGrad$-Lipschitz first derivatives and $\LipHess$-Lipschitz second derivatives on the set $\tilde{\X}\ind{P2}$.
\end{assumption}

Before presenting Claim~\ref{ClaimTwoPhase} let us introduce non-negative scalars $c$, $\Delta_{\obj}$, and $\Delta_{\cons}$ chosen as follows.
\begin{subequations}\label{assume-two-phase-constants}
	\begin{flalign}
		c &\ge \sup_{x \in \tilde{\X}\ind{P1}} \max_{i \in \ConSet} \cons_i(x) \label{assume-two-phase-constants-c} \\
		\Delta_{\obj} &\ge \sup_{z \in \tilde{\X}\ind{P2}} \obj(z) - \inf_{z \in \tilde{\X}\ind{P2}}{\obj(z)}  \\
		\Delta_{\cons} &\ge \min_{i \in \ConSet} \max\{-\cons_i(x\ind{0}),0\}.
	\end{flalign}
\end{subequations}

\begin{restatable}{myclaim}{ClaimTwoPhase}\label{ClaimTwoPhase}
	~~Let $x\ind{0} \in \reals^{\NumVar}$.
	Suppose Assumption~\ref{assume-lip-deriv-two-phase} and \eqref{assume-two-phase-constants} holds. 
	Let $f$ be $\LipFun$-Lipschitz. 
	Assume $c, \Delta_{\cons}, \Delta_{\obj}, \LipGrad, \LipFun \ge 1$, $\epsInf \in (0, \frac{\editThree{1}}{\NumCon}]$ and $\epsOpt \in (0,\sqrt{\epsInf}] \cap \Big(0,\frac{1}{\NumCon \plusLog(c / \epsOpt)} \Big] \editThree{\cap \Big(0,\frac{\LipGrad^3}{\LipHess^2} \Big]}$.
	%Assume \eqref{eq:small-eps-inf-opt} holds.
	Then \callTwoPhase{$\obj, \cons, \epsOpt, \epsInf, \LipFun, \LipGrad, \x\ind{0}$} takes at most
	$$
	\bigO{ \Delta_{\cons}  \left( \frac{\LipGrad^{3/4}}{\epsInf^{7/4}  \epsOpt^{1/4}}  + \frac{1}{\epsInf \epsOpt} \right)  + \frac{\Delta_{\obj}}{\epsOpt} \left(\frac{\LipGrad \LipFun}{\epsOpt \epsInf} \right)^{3/4} }
	$$
	trust-region subproblem solves to return a point $(x,t,y)$ that satisfies either \eqref{eq:KKT} or \eqref{eq:Linf-infeas}.
\end{restatable}
\CHECKED

The definition of \callTwoPhase{} appears in Section~\ref{sec:algTwoPhase} and the proof of Claim~\ref{ClaimTwoPhase} appears in Section~\ref{sec:ClaimTwoPhase}. The proof is primarily devoted to analyzing phase-two when we minimize the objective while approximately satisfying the constraints. We argue that when we terminate with an approximate Fritz John point in phase-two then either the dual variables are small enough that this is a KKT point or if the dual variables are large the scaled dual variables give an infeasibility certificate. If we add the assumption that $\epsOpt \in (0, \epsInf]$ the iteration bound of Claim~\ref{ClaimTwoPhase} can be even more simply stated as
\begin{flalign}\label{eq:super-simple-bound}
	\bigO{ \frac{\Delta_{\cons} + \Delta_{\obj}}{\epsOpt} \left(\frac{\LipGrad \LipFun}{\epsOpt \epsInf} \right)^{3/4} }.
\end{flalign}
We can now compare with the results of \cite{birgin2016evaluation} in Table~\ref{table:compare-birgin}. 

\begin{table}[H]
	\begin{tabular}{ c | c | c | c | c   }
		algorithm & \# iteration & iteration subproblem  &  evaluates & Lipschitz   \\ 
		\hline
		%%%%%%%
		\citet[$p=1$]{birgin2016evaluation} & $\bigO{ \epsOpt^{-3} \epsInf^{-2} }$ 
		& gradient computation & $f,a, \grad$ & $f,a, \grad, \grad^2$ \\  
		%%%%%%%
		%\cite{cartis2011evaluation}  & $\bigO{ \epsOpt^{-2} }*$ 
		%& linear program &  $\grad$ \\  
		%%%%%%%
		\citet[$p=2$]{birgin2016evaluation} & $\bigO{ \epsOpt^{-2} \epsInf^{-3/2} }$ 
		&  CRN with non-negativity constraint & $f,a, \grad$, $\grad^2$ & $f,a, \grad, \grad^2$ \\  
		%%%%%%%
		IPM (this paper) & $\bigO{ \epsOpt^{-7/4} \epsInf^{-3/4} }$ 
		& trust-region subproblem & $\grad$, $\grad^2$ & $\grad, \grad^2$ \\
		%My thesis & $\bigO{ \epsOpt^{-3/2} \epsInf^{-1/2} }$ & QCQP & $\grad$, $\grad^2$  
	\end{tabular}
	\caption{This table compares iteration bounds under the setup of \eqref{eq:super-simple-bound}. It only includes dependencies on $\epsOpt$ and $\epsInf$. CRN stands for cubic regularized Newton \cite{nesterov2006cubic}.
	}\label{table:compare-birgin}
\end{table}

The algorithm of \citet{birgin2016evaluation} sequentially finds KKT points to quadratic penalty subproblems of the form,
\begin{flalign}\label{eq:birgin:subproblem}
	\minimize_{(x,r,s) \in \reals^{\NumVar + 1 + \NumCon}} \Phi_{t}(x,r,s) := (f(x) - t + r)^{2} + \| a(x) + s \|_2^{2} \quad \text{s.t.} \quad r \ge 0, \quad s \ge \zeros.
\end{flalign}
To solve this subproblem method they suggest using $p$th order regularization with non-negativity constraints. For $p=1$ this reduces to projected gradient descent which has low per-iteration cost but results in unfavorable iteration bounds in terms of $\epsOpt$ and $\epsInf$. For $p=2$ this reduces to cubic regularization Newton's method with non-negativity constraints, i.e.,
\begin{flalign}\label{eq:birgin:cubic-subproblem}
	\minimize_{\Dir \in \reals^{\NumVar+ 1 + \NumCon}} \frac{1}{2} \Dir^T \grad^2 \Phi_{t}(x,r,s) \Dir + \grad \Phi_{t}(x,r,s)^T \Dir + C \| d \|_2^3 \quad \text{s.t} \quad r + \dir{r} \ge 0, \quad s + \dir{s} \ge \zeros
\end{flalign}
for some constant $C > 0$ with $d = (\dir{x}, \dir{r}, \dir{s})$. Solving this subproblem might be computationally expensive. It is well-known that checking if a point is a local optimum of \eqref{eq:birgin:cubic-subproblem} is in general NP-hard \cite{pardalos1988checking}. It is possible to find an approximate KKT point using projected gradient descent or an interior point method for solving nonconvex quadratic program \cite{ye1998complexity}. However, both these approaches are likely to result in a computational runtime with worse $\epsOpt$ and $\epsInf$ dependence than $\bigO{ \epsOpt^{-2} \epsInf^{-3/2} }$. We speculate that one might also be able to apply the interior point method of \citet{haeser2019optimality} as the unconstrained minimization algorithm for solving \eqref{eq:birgin:subproblem} and potentially obtain the runtime bound of $\bigO{ \epsOpt^{-2} \epsInf^{-3/2} }$ given by \cite{birgin2016evaluation}, although further analysis is needed to confirm this. 

Finally, \citet{cartis2011evaluation,cartis2014complexity} show that one requires $\bigO{ \epsOpt^{-2} }$ iterations to find a \textbf{scaled} KKT point:
\[
\| \grad_{x} \Lag(x,y) \|_{2} \le \epsOpt (\| y \|_2 + 1), \quad y \ge \zeros, \quad \quad \cons(x) \ge -\epsOpt \ones, \quad \cons_i(x) y_i \le \epsOpt (1 + \| y \|_2) \quad \forall i \in \ConSet,
\]
or a certificate of infeasibility. Their method only requires computation of first-derivatives but has the disadvantage that it requires solving a linear program at each iteration.
Recently, this approach was extended to arbitrary higher-order derivatives to obtain an $\bigO{\epsOpt^{-(p+1)/p}}$ iteration bound \cite{cartis2020strong}. For $p = 2$ this yields an iteration bound of $\bigO{\epsOpt^{-3/2}}$. 
There are two caveats to this result: each iteration requires evaluating the first and second derivatives of the objective and constraints, and solving an expensive subproblem (a quadratically constrained quadratic program). 
In contrast, each iteration of our method consists of evaluating the gradient and Hessian of the Lagrangian, and exactly solving a trust-region subproblem.
In \editThree{the exact arithmetic model of computation}, this
trust-region subproblem can be \editThree{exactly solved} in $O(n^3)$ arithmetic operations
 (Remark~\ref{remark:solving-trust-region-subproblem}).
}

\newpage

%%%%%%%%%%%%%%%%%%%%%%%%%%%%%%%%%%%%%%%%%%%%%%%%%%
%%%%%%%%%%%%%%%%%%%%%%%%%%%%%%%%%%%%%%%%%%%%%%%%%%
%%%%%%%%%%%%%%%%%%%%%%%%%%%%%%%%%%%%%%%%%%%%%%%%%%
\ifMOR
\begin{APPENDICES}
\else
\appendix
\fi 
%%%%%%%%%%%%%%%%%%%%%%%%%%%%%%%%%%%%%%%%%%%%%
%%%%%%%%%%%%%%%%%%%%%%%%%%%%%%%%%%%%%%%%%%%%%%%%%%
%%%%%%%%%%%%%%%%%%%%%%%%%%%%%%%%%%%%%%%%%%%%%%%%%%
%%%%%%%%%%%%%%%%%%%%%%%%%%%%%%%%%%%%%%%%%%%%%%%%%%

%%%%%%%%%%
%% TWO PHASE
%%%%%%%%%%
\edit{
\section{A two-phase method to find unscaled KKT points}\label{sec:TwoPhase}

\subsection{Algorithm~\ref{alg:twoPhase} definition}\label{sec:algTwoPhase}

~
\begin{algorithm}[H]
	\caption{Two-phase IPM}\label{alg:twoPhase}
	\begin{algorithmic}
		\Function{\AlgTwoPhase}{$\obj, \cons, \epsOpt, \epsInf, \LipFun, \LipGrad, \x\ind{0}$}
		\State \textbf{Output:} A status (\textsc{KKT} if \eqref{eq:KKT} holds and \textsc{INF} if \eqref{eq:Linf-infeas} holds) and a point $(x,t,y)$.
		\State
		\State \textbf{Phase-one.} 
		\State Let $\mu\ind{P1} = \frac{\epsInf \epsOpt}{12}$, $\tau_{l}\ind{P1} = \min\left\{ \frac{1}{\epsOpt}, \sqrt{\frac{\LipGrad}{2 \epsOpt \epsInf} } \right\}$, $\editThree{\tau_{c}\ind{P1} = \tau_{l}\ind{P1} \left(\frac{\mu\ind{P1}}{\LipGrad}\right)^{1/2}}$, $t\ind{0} = \frac{\epsOpt}{2} + \max\{\min_{i \in \ConSet} -\cons_i(x\ind{0}), 0\}$, and $\eta$ satisfy \eqref{eq:choose-eta-nonconvex}.
		\If{$t\ind{0} \le \epsOpt/2$}
		\State $\x\ind{P1} \gets \x\ind{0}$
		\Else
		\State $(\x\ind{P1}, t\ind{P1}, y\ind{P1}, \lambda\ind{P1}, \gamma\ind{P1}) \gets \callAlgMain{\obj\ind{P1}, \cons\ind{P1}, \mu\ind{P1}, \tau_{l}\ind{P1}, \tau_{c}\ind{P1}, \LipGrad, \edit{\eta}, (\x\ind{0},t\ind{0})}.$
		% \vspace{0.3cm}
		\If{ $\min_{i \in \ConSet} \cons_i(x\ind{P1}) < -\epsOpt/2$ } % \eqref{eq:infeasibility} is satisfied with 
		\State $(x,t,y) \gets (x\ind{P1}, t\ind{P1}, \y\ind{P1} /\| \y\ind{P1} \|_1).$
		\State \Return{\textsc{INF}, $(x,t,y)$}
		\EndIf
		\EndIf
		\State
		\State \textbf{Phase-two.} 
		\State Let $\mu\ind{P2} = \frac{\epsOpt}{4}$, $\tau_{l}\ind{P2} = \sqrt{ \frac{\epsInf}{2(\LipFun+1)} }$, $\editThree{\tau_{c}\ind{P2} = \tau_{l}\ind{P2} \left(\frac{\mu\ind{P2}}{\LipGrad}\right)^{1/2}}$, and $\eta$ satisfy \eqref{eq:choose-eta-nonconvex}.
		\State $(x\ind{P2}, y\ind{P2}) \gets \callAlgMain{\obj, \cons\ind{P2}, \mu\ind{P2}, \tau_{l}\ind{P2}, \tau_{c}\ind{P2}, \LipGrad, \edit{\eta}, \x\ind{P1}}.$ 
		\If{$\| y\ind{P2} \|_1 > 1 / \epsInf$}
		\State  $(x,t,y) \gets (x\ind{P2}, \epsOpt, y\ind{P2} /\| y\ind{P2} \|_1)$
		\State \Return{\textsc{INF}, $(x,t,y)$}
		\Else
		\State $(x,t,y) \gets (x\ind{P2}, \emptyset, y\ind{P2}).$
		\State \Return{\textsc{KKT}, $(x,t,y)$}
		\EndIf
		\EndFunction
	\end{algorithmic}
\end{algorithm}

\subsection{Proof of Claim~\ref{ClaimTwoPhase}}\label{sec:ClaimTwoPhase}

\ifRepeatThms
\ClaimTwoPhase*
\fi

\begin{myproof}
	Let $\psi_{\mu\ind{P2}}\ind{P1}$ and $\psi_{\mu\ind{P2}}\ind{P2}$ denote the log barrier for problems \eqref{eq:phase-one-new} and \eqref{eq:phase-two-new} respectively. Let $T \defeq [0, t\ind{0} + \epsOpt/2]$ represent the set of feasible values of $t$ in phase-one.
	Now,
	\begin{flalign*}
		&\psi_{\mu\ind{P1}}\ind{P1}(x\ind{0}, t\ind{0}) - \inf_{(x,t) \in \tilde{\X}\ind{P1} \times T} \psi_{\mu\ind{P1}}\ind{P1}(x,t)  \\
		%%%%%%%%%%%%%%%%%%%%%%%%%%%%
		&= \sup_{(x,t) \in \tilde{\X}\ind{P1} \times T} t\ind{0} - t + \mu\ind{P1} \left( \log\left( \frac{t}{t\ind{0}} \right) + \log\left(\frac{\frac{\epsOpt}{2} + t\ind{0} - t}{\frac{\epsOpt}{2}} \right) + \sum_{i \in \ConSet} \log\left( \editThree{\frac{\cons_i(x) + t }{\cons_i(x\ind{0}) + t\ind{0}}} \right) \right) \\
		%%%%%%%%%%%%%%%%%%%%%%%%%%%%
		&\editThree{\le t\ind{0} + \mu\ind{P1} \left( \log\left( \frac{t\ind{0} + \frac{\epsOpt}{2}}{2 t\ind{0}} \right) + \log\left( \frac{t\ind{0} + \frac{\epsOpt}{2}}{\epsOpt} \right) + \sup_{x \in \tilde{\X}\ind{P1}} \sum_{i \in \ConSet} \log\left( \editThree{\frac{\cons_i(x) + t }{ \cons_i(x\ind{0}) + t\ind{0}}} \right) \right)} \\
		%%%%%%%%%%%%%%%%%%%%%%%%%%%%
		&\editThree{\le t\ind{0} + \mu\ind{P1} \left( \bigO{ \plusLog\left( \frac{c}{\epsOpt} \right) } + \sup_{x \in \tilde{\X}\ind{P1}} \sum_{i \in \ConSet} \log\left( \editThree{\frac{\cons_i(x) + t }{ \cons_i(x\ind{0}) + t\ind{0}}} \right) \right)} \\
		%%%%%%%%%%%%%%%%%%%%%%%%%%%%
		&= \bigO{\editThree{\epsOpt} + \min_{i \in \ConSet} \max\{-\cons_i(x\ind{0}), 0\} + \mu\ind{P1} \NumCon \plusLog(c/\epsOpt)}  \\
		%%%%%%%%%%%%%%%%%%%%%%%%%%%%
		&= \bigO{\Delta_{\cons}}
	\end{flalign*}
	where the \editThree{second transition uses the inequality $\log(\theta) + \log(R - \theta) \le 2 \log(R / 2)$ for all $R > 0$ and $\theta \in (0,R)$ with $\theta = t$ and $R = t\ind{0} + \frac{\epsOpt}{2}$, the third transition uses that $t\ind{0} \le c + \epsOpt / 2$ by \eqref{assume-two-phase-constants-c}, the fourth} transition uses that $0 \le t \le t\ind{0} + \epsOpt/2 = \min_{i \in \ConSet} \max\{-\cons_i(x\ind{0}), 0\} + \epsOpt$ and \eqref{assume-two-phase-constants-c}, and the last transition uses $\mu\ind{P1} = \frac{1}{12} \epsInf \epsOpt = \bigO{ \epsOpt }$, $\epsOpt \in \Big(0,\frac{1}{\NumCon \plusLog(c / \epsOpt)} \Big]$ and $\Delta_{\cons} \ge 1$.
	
	Similarly, using $\mu\ind{P2} =  \epsOpt / 4$, $\epsOpt \in \Big(0,\frac{1}{\NumCon \plusLog(c / \epsOpt)} \Big]$, $\Delta_{\obj} \ge 1$ \editThree{and \eqref{assume-two-phase-constants-c} (recall $\tilde{\X}\ind{P2} \subseteq \tilde{\X}\ind{P1}$)} we get
	\begin{flalign*}
		\psi_{\mu\ind{P2}}\ind{P2}(x\ind{P1}) - \inf_{x \in \tilde{\X}\ind{P2}} \psi_{\mu\ind{P2}}\ind{P2}(x) &= \bigO{ f(x\ind{P1}) - \inf_{x \in \tilde{\X}\ind{P2}} f(x) + \mu\ind{P2} \NumCon \plusLog(c/\epsOpt)} \\
		&= \bigO{\Delta_{\obj}}.
	\end{flalign*}
\editThree{Next, we verify we can employ Theorem~\ref{thmMainResultNonconvex} to analyze Algorithm~\ref{alg:twoPhase} by confirming Assumption~\ref{assumption:nonconvex-pars} holds. In particular,
employing the assumptions on $\LipFun, \LipGrad, \epsOpt$, $\epsInf$ given in the premise of Claim~\ref{ClaimTwoPhase} to values of $\tau_{l}$ and $\mu$ defined in Algorithm~\ref{alg:twoPhase} we get
\begin{flalign*}
\frac{(\tau_{l}\ind{P1})^2 \mu\ind{P1}}{\LipGrad} &\le \frac{\frac{\LipGrad}{2 \epsOpt \epsInf} \times \frac{\epsInf \epsOpt}{12}}{\LipGrad} = \frac{1}{24} \le 1 \\
%%%%%%%%%%%%%%%%%%%%%%%%%%
\frac{\LipHess^2 \mu\ind{P1}}{\LipGrad^3} &= \frac{\LipHess^2 \epsOpt \epsInf}{12 \LipGrad^3} \le \frac{\epsInf}{12} \le 1 \\
%%%%%%%%%%%%%%%%%%%%%%%%%%
\frac{\LipHess^2 \mu\ind{P2}}{\LipGrad^3} &= \frac{\LipHess^2 \epsOpt}{4 \LipGrad^3} \le \frac{1}{4} \le 1 \\
%%%%%%%%%%%%%%%%%%%%%%%%%%
\frac{(\tau_{l}\ind{P2})^2 \mu\ind{P2}}{\LipGrad} &= \frac{\frac{\epsInf}{2 (\LipFun + 1)} \times \frac{\epsOpt}{4}}{\LipGrad} \le \frac{\epsOpt \epsInf}{16} \le 1.
\end{flalign*}}
Recall Theorem~\ref{thmMainResultNonconvex} gives a bound on the iteration count of $\callAlgMain{}$ of $\bigO{1 + \frac{\barrier(x\ind{0}) - \barrier^{*}}{\mu} \left(\frac{\LipGrad}{\mu \tau_{l}^2} \right)^{3/4}}$. Substituting the values of $\tau_{l}$ and $\mu$ from Algorithm~\ref{alg:twoPhase} yields a bound of
	$$
	\bigO{ 1 +  \Delta_{\cons}  \left( \frac{\LipGrad^{3/4}}{\epsInf^{7/4}  \epsOpt^{1/4}}  + \frac{1}{\epsInf \epsOpt} \right)  + \frac{\Delta_{\obj}}{\epsOpt} \left(\frac{\LipGrad \LipFun}{\epsOpt \epsInf} \right)^{3/4} }
	$$
	trust-region subproblem solves for $\callTwoPhase{}$.
	
	It remains to show that either \eqref{eq:KKT} or \eqref{eq:Linf-infeas} is satisfied. 
	\editThree{We start by analyzing phase-one and show if it does not produce an approximately 
	feasible solution then it produces a point satisfying the infeasibility criteria \eqref{eq:Linf-infeas}.}
	Observe that after calling $\callAlgMain{}$ in phase-one we find a point satisfying the approximate Fritz John conditions for the problem of minimizing the infinity norm of the constraint violation, i.e.,
	\begin{flalign}
		& \left\| 
		\begin{pmatrix} \grad \cons(\x\ind{P1})^T y\ind{P1} \\ 
			\ones^T \y\ind{P1} - 1 + \lambda\ind{P1}- \gamma\ind{P1} \\
		\end{pmatrix}  \right\|_2 
		\le \frac{\epsInf}{12} \sqrt{ \left\| \begin{pmatrix} y\ind{P1} \\ \lambda\ind{P1} \\  \gamma\ind{P1}  \end{pmatrix} \right\|_1 + 1}  \label{eq:phase-one:dual-feas} \\
		& a(x\ind{P1}) + t\ind{P1} \ones \ge \zeros \label{eq:phase-one:a-plus-t-pos} \\
		& 0 \le t\ind{P1} \le t\ind{0} + \epsOpt/2 \label{eq:phase-one:t-nonnegative} \\
		& (a_i(x\ind{P1}) + t\ind{P1}) y\ind{P1}_i  \le \frac{1}{6} \epsInf \epsOpt \label{eq:phase-one:mu} \\
		& t\ind{P1}  \lambda\ind{P1}  \le \frac{1}{6} \epsInf \epsOpt  \label{eq:phase-one:lambda} \\
		& \left( \frac{\epsOpt}{2} + t\ind{0} - t\ind{P1} \right) \gamma\ind{P1}  \le \frac{1}{6} \epsInf \epsOpt  \label{eq:phase-one:gamma} \\
		& (y\ind{P1}, \lambda\ind{P1}, \gamma\ind{P1}) \ge \zeros.
	\end{flalign}
	\editThree{Consider the case that $\min_{i \in \ConSet} \cons_i(x\ind{P1}) < -\epsOpt/2$ in which case phase-one returns the} status $\textsc{INF}$. Consequently, $t\ind{P1} > \epsOpt/2$ by \eqref{eq:phase-one:a-plus-t-pos}. Using $t\ind{P1} > \epsOpt/2$ and \eqref{eq:phase-one:lambda} we deduce $\lambda\ind{P1} < \frac{\epsInf}{3}$.
	Therefore using \eqref{eq:phase-one:dual-feas}, $\epsInf \in (0,1]$ and we deduce
	\begin{flalign}\label{eq:313120318675}
		\left\| \begin{pmatrix} \grad \cons(\x\ind{P1})^T y\ind{P1} \\ \ones^T \y\ind{P1} - 1 - \gamma\ind{P1} \end{pmatrix}  \right\|_2 \le \frac{\epsInf}{12} \sqrt{ \| y\ind{P1} \|_1 + \frac{\epsInf}{3} + 1}  \le
		\frac{\epsInf}{12} \left(\sqrt{ \| y\ind{P1} \|_1 } + 2 \right). 
	\end{flalign}
	
	If $\| y\ind{P1} \|_1 < 1/2$ then using \eqref{eq:313120318675} we deduce $1/2 < \gamma\ind{P1} + 1 -  \ones^T \y\ind{P1} \le \epsInf / 2 \le 1/2$. By contradiction $\| y\ind{P1} \|_1 \ge 1/2$. Using $\| y\ind{P1} \|_1 \ge 1/2$, \eqref{eq:313120318675}, and \eqref{eq:phase-one:mu} we deduce
	$$
	\frac{\| \grad \cons(\x\ind{P1})^T y\ind{P1} \|_2}{ \| y\ind{P1} \|_1 } \le \epsInf \quad  \frac{(a_i(x\ind{P1}) + t\ind{P1}) y\ind{P1}_i}{ \| y\ind{P1} \|_1 } \le \epsInf \epsOpt.
	$$
	\editThree{Therefore, \eqref{eq:Linf-infeas} holds.}

	\editThree{Finally, we prove after phase-two \eqref{eq:KKT} or \eqref{eq:Linf-infeas} is satisfied.}
	Observe that after calling $\callAlgMain{}$ in phase-two we find a point satisfying
	%\begin{taggedsubequations}{FJ1}\label{eq:first-order-SIP-full:first-order}%\tag{FJ1}
	\begin{flalign*}
		\cons(\x\ind{P2}) &> -\epsOpt \ones \\
		\y_i\ind{P2} (\cons_i(\x\ind{P2}) + \epsOpt) &\le  \frac{1}{2} \epsOpt  \quad \forall i \in \ConSet  \\
		\norm{\grad_{x} \Lag(\x\ind{P2},\y\ind{P2}) }_{2} &\le \frac{\epsOpt}{4} \sqrt{ \frac{\epsInf}{2(\LipFun+1)} } \sqrt{\norm{\y\ind{P2}}_{1} + 1} \\
		\y\ind{P2} &>  \zeros.
	\end{flalign*}
	If $\| \y\ind{P2} \|_1 < \frac{3 \epsOpt^2}{\epsInf^2} + \frac{3 \LipFun}{\epsInf}$ then using the fact that $\epsOpt \in (0,1]$, $\epsOpt \in (0,\sqrt{\epsInf}]$ and $\LipFun \ge 1$ we get
	$$
	\norm{\grad_{x} \Lag(\x\ind{P2},\y\ind{P2}) }_{2} \le \frac{\epsOpt}{4} \sqrt{ \frac{\epsInf}{2\LipFun} \left( \frac{3\epsOpt^2}{\epsInf^2} + \frac{3\LipFun}{\epsInf} \right) + 1} \editThree{= \frac{\epsOpt}{4} \sqrt{ \frac{3}{2}\frac{\epsOpt^2}{\epsInf \LipFun} + \frac{5}{2}}} \le  \frac{\epsOpt}{2}
	$$
	which implies \eqref{eq:KKT} is satisfied. Otherwise if $\| \y\ind{P2} \|_1 \ge  \frac{3\epsOpt^2}{\epsInf^2} + \frac{3\LipFun}{\epsInf}$ then \editThree{using that $\epsInf \in (0, 1]$ and $\LipFun \ge 1$ we have}
	$$
	\frac{\| \grad \cons(x\ind{P2})^T y\ind{P2} \|_2}{\| y\ind{P2} \|_1} \le \frac{\norm{\grad_{x} \Lag(\x\ind{P2},\y\ind{P2}) }_{2}  + \norm{\grad f(x\ind{P2}) }_2}{\| y\ind{P2} \|_1} \le \frac{\epsOpt}{\| y\ind{P2} \|_1^{1/2}} + \frac{\epsOpt}{\| y\ind{P2} \|_1} + \frac{\LipFun}{\| y\ind{P2} \|_1} \le \epsInf
	$$
	and \editThree{as $\LipFun \ge 1$,}
	$$
	\frac{(a_i(x\ind{P2}) + \epsOpt) y\ind{P2}_i}{ \| y\ind{P2} \|_1 } \le \epsInf \epsOpt.
	$$
	Finally note that since $\y_i\ind{P2} (\cons_i(\x\ind{P2}) + \epsOpt) \le  \frac{1}{2} \epsOpt$ and $\| \y\ind{P2} \|_1 \ge  \frac{3\epsOpt^2}{\epsInf^2} + \frac{3\LipFun}{\epsInf} \ge \NumCon$ we deduce $\min_{i \in \ConSet} \cons_i(x\ind{P2}) \le \epsOpt \min_{i \in \ConSet} \left( \frac{1}{2 y\ind{P2}_i} - 1 \right) \le -\epsOpt/2$. Hence \eqref{eq:Linf-infeas} is satisfied with $(x,t,y) = \left(x\ind{P2}, \epsOpt, \frac{y\ind{P2}}{\| y\ind{P2} \|_1}\right)$.
\end{myproof}
}

\arxiv{
\bibliographystyle{apalike}
\bibliography{IPM-trust}
}

\section*{Acknowledgements}

We'd like to thank Yair Carmon, Ron Estrin, Michael Saunders and the anonymous reviewer for their useful feedback on the paper. 
The first author was primarily supported by the PACCAR Inc
Stanford Graduate Fellowship and the Dantzig-Lieberman fellowship for this work.
The first author also recieved partial support from AFOSR grant \#FA$9550$-$23$-$1$-$0242$.
\ifMOR
\end{APPENDICES}
\fi 

\ifMOR
\bibliographystyle{informs2014}
\bibliography{IPM-trust.bib}

\begin{thebibliography}{}

\bibitem[Andersen, 2001]{andersen2001certificates}
Andersen, E.~D. (2001).
\newblock Certificates of primal or dual infeasibility in linear programming.
\newblock {\em Computational Optimization and Applications}, 20:171--183.

\bibitem[Andersen and Ye, 1998]{andersen1998computational}
Andersen, E.~D. and Ye, Y. (1998).
\newblock A computational study of the homogeneous algorithm for large-scale
  convex optimization.
\newblock {\em Computational Optimization and Applications}, 10(3):243--269.

\bibitem[Bian et~al., 2015]{bian2015complexity}
Bian, W., Chen, X., and Ye, Y. (2015).
\newblock Complexity analysis of interior point algorithms for non-{L}ipschitz
  and nonconvex minimization.
\newblock {\em {M}athematical {P}rogramming}, 149(1-2):301--327.

\bibitem[Birgin et~al., 2016]{birgin2016evaluation}
Birgin, E.~G., Gardenghi, J., Mart{\'\i}nez, J., Santos, S., and Toint, P.~L.
  (2016).
\newblock Evaluation complexity for nonlinear constrained optimization using
  unscaled {KKT} conditions and high-order models.
\newblock {\em SIAM Journal on Optimization}, 26(2):951--967.

\bibitem[Byrd et~al., 2000]{byrd2000trust}
Byrd, R.~H., Gilbert, J.~C., and Nocedal, J. (2000).
\newblock A trust region method based on interior point techniques for
  nonlinear programming.
\newblock {\em {M}athematical {P}rogramming}, 89(1):149--185.

\bibitem[Byrd et~al., 2006]{byrd2006knitro}
Byrd, R.~H., Nocedal, J., and Waltz, R.~A. (2006).
\newblock {KNITRO}: An integrated package for nonlinear optimization.
\newblock In {\em Large-scale nonlinear optimization}, pages 35--59. Springer.

\bibitem[Carmon et~al., 2020]{carmon2020lowerI}
Carmon, Y., Duchi, J.~C., Hinder, O., and Sidford, A. (2020).
\newblock Lower bounds for finding stationary points {I}.
\newblock {\em {M}athematical {P}rogramming}, 184(1-2):71--120.

\bibitem[Carmon et~al., 2021]{carmon2021lowerII}
Carmon, Y., Duchi, J.~C., Hinder, O., and Sidford, A. (2021).
\newblock Lower bounds for finding stationary points {II}: first-order methods.
\newblock {\em {M}athematical {P}rogramming}, 185(1-2):315--355.

\bibitem[Cartis et~al., 2020]{cartis2020strong}
Cartis, C., Gould, N., and Toint, P.~L. (2020).
\newblock Strong evaluation complexity bounds for arbitrary-order optimization
  of nonconvex nonsmooth composite functions.
\newblock {\em arXiv preprint arXiv:2001.10802}.

\bibitem[Cartis et~al., 2011a]{cartis2011adaptive}
Cartis, C., Gould, N.~I., and Toint, P.~L. (2011a).
\newblock Adaptive cubic regularisation methods for unconstrained optimization.
  part ii: worst-case function-and derivative-evaluation complexity.
\newblock {\em {M}athematical {P}rogramming}, 130(2):295--319.

\bibitem[Cartis et~al., 2011b]{cartis2011evaluation}
Cartis, C., Gould, N.~I., and Toint, P.~L. (2011b).
\newblock On the evaluation complexity of composite function minimization with
  applications to nonconvex nonlinear programming.
\newblock {\em SIAM Journal on Optimization}, 21(4):1721--1739.

\bibitem[Cartis et~al., 2014]{cartis2014complexity}
Cartis, C., Gould, N.~I., and Toint, P.~L. (2014).
\newblock On the complexity of finding first-order critical points in
  constrained nonlinear optimization.
\newblock {\em {M}athematical {P}rogramming}, 144(1-2):93--106.

\bibitem[Chen and Goldfarb, 2006]{chen2006interior}
Chen, L. and Goldfarb, D. (2006).
\newblock Interior-point $l_2$-penalty methods for nonlinear programming with
  strong global convergence properties.
\newblock {\em {M}athematical {P}rogramming}, 108(1):1--36.

\bibitem[Conn et~al., 2000a]{conn2000primal}
Conn, A.~R., Gould, N.~I., Orban, D., and Toint, P.~L. (2000a).
\newblock A primal-dual trust-region algorithm for non-convex nonlinear
  programming.
\newblock {\em {M}athematical {P}rogramming}, 87(2):215--249.

\bibitem[Conn et~al., 2000b]{conn2000trust}
Conn, A.~R., Gould, N.~I., and Toint, P.~L. (2000b).
\newblock {\em Trust region methods}.
\newblock SIAM.

\bibitem[Curtis et~al., 2017]{curtis2017trust}
Curtis, F.~E., Robinson, D.~P., and Samadi, M. (2017).
\newblock A trust region algorithm with a worst-case iteration complexity of
  $\mathcal{O}(\epsilon^{-3/2})$ for nonconvex optimization.
\newblock {\em {M}athematical {P}rogramming}, 162(1-2):1--32.

\bibitem[Gander et~al., 1989]{gander1989constrained}
Gander, W., Golub, G.~H., and Von~Matt, U. (1989).
\newblock A constrained eigenvalue problem.
\newblock {\em Linear Algebra and its applications}, 114:815--839.

\bibitem[Ghojogh et~al., 2019]{ghojogh2019eigenvalue}
Ghojogh, B., Karray, F., and Crowley, M. (2019).
\newblock Eigenvalue and generalized eigenvalue problems: Tutorial.
\newblock {\em arXiv preprint arXiv:1903.11240}.

\bibitem[Gould et~al., 2015]{gould2015interior}
Gould, N.~I., Orban, D., and Toint, P.~L. (2015).
\newblock An interior-point $\ell_1$-penalty method for nonlinear optimization.
\newblock In {\em Numerical Analysis and Optimization}, pages 51--190. Springer
  International Publishing.

\bibitem[Gould et~al., 2010]{gould2010solving}
Gould, N.~I., Robinson, D.~P., and Thorne, H.~S. (2010).
\newblock On solving trust-region and other regularised subproblems in
  optimization.
\newblock {\em {M}athematical {P}rogramming {C}omputation}, 2(1):21--57.

\bibitem[Haeser et~al., 2019]{haeser2019optimality}
Haeser, G., Liu, H., and Ye, Y. (2019).
\newblock Optimality condition and complexity analysis for linearly-constrained
  optimization without differentiability on the boundary.
\newblock {\em {M}athematical {P}rogramming}, 178:263--299.

\bibitem[Hinder and Ye, 2018]{hinder2018one}
Hinder, O. and Ye, Y. (2018).
\newblock A one-phase interior point method for nonconvex optimization.
\newblock {\em arXiv preprint arXiv:1801.03072}.

\bibitem[John, 1948]{john1948extremum}
John, F. (1948).
\newblock Extremum problems with inequalities as side conditions.
\newblock {\em Studies and Essays: Courant Anniversary Volume}, pages 187--204.

\bibitem[Karmarkar, 1984]{karmarkar1984new}
Karmarkar, N. (1984).
\newblock A new polynomial-time algorithm for linear programming.
\newblock In {\em Proceedings of the sixteenth annual {ACM} symposium on Theory
  of computing}, pages 302--311. ACM.

\bibitem[Kojima et~al., 1989a]{kojima1989polynomial}
Kojima, M., Mizuno, S., and Yoshise, A. (1989a).
\newblock A polynomial-time algorithm for a class of linear complementarity
  problems.
\newblock {\em {M}athematical {P}rogramming}, 44(1):1--26.

\bibitem[Kojima et~al., 1989b]{kojima1989primal}
Kojima, M., Mizuno, S., and Yoshise, A. (1989b).
\newblock A primal-dual interior point algorithm for linear programming.
\newblock In {\em Progress in {M}athematical {P}rogramming}, pages 29--47.
  Springer.

\bibitem[Mangasarian and Fromovitz, 1967]{mangasarian1967fritz}
Mangasarian, O.~L. and Fromovitz, S. (1967).
\newblock The fritz john necessary optimality conditions in the presence of
  equality and inequality constraints.
\newblock {\em Journal of Mathematical Analysis and Applications},
  17(1):37--47.

\bibitem[Megiddo, 1989]{megiddo1989pathways}
Megiddo, N. (1989).
\newblock Pathways to the optimal set in linear programming.
\newblock In {\em Progress in {M}athematical {P}rogramming}, pages 131--158.
  Springer.

\bibitem[Mehrotra, 1992]{mehrotra1992implementation}
Mehrotra, S. (1992).
\newblock On the implementation of a primal-dual interior point method.
\newblock {\em SIAM Journal on optimization}, 2(4):575--601.

\bibitem[Monteiro and Adler, 1989]{monteiro1989interior}
Monteiro, R.~D. and Adler, I. (1989).
\newblock Interior path following primal-dual algorithms. part {I}: Linear
  programming.
\newblock {\em {M}athematical {P}rogramming}, 44(1):27--41.

\bibitem[Nemirovskii and Yudin, 1983]{nemirovskii1983problem}
Nemirovskii, A. and Yudin, D.~B. (1983).
\newblock Problem {C}omplexity and {M}ethod {E}fficiency in {O}ptimization.

\bibitem[Nesterov, 1983]{nesterov1983method}
Nesterov, Y. (1983).
\newblock A method of solving a convex programming problem with convergence
  rate ${O}(1/k^2)$.
\newblock In {\em Soviet Mathematics Doklady}, volume~27, pages 372--376.

\bibitem[Nesterov and Nemirovskii, 1994]{nesterov1994interior}
Nesterov, Y. and Nemirovskii, A. (1994).
\newblock {\em Interior-point polynomial algorithms in convex programming}.
\newblock SIAM.

\bibitem[Nesterov and Polyak, 2006]{nesterov2006cubic}
Nesterov, Y. and Polyak, B.~T. (2006).
\newblock Cubic regularization of newton method and its global performance.
\newblock {\em {M}athematical {P}rogramming}, 108(1):177--205.

\bibitem[Nocedal and Wright, 2006]{nocedal2006numerical}
Nocedal, J. and Wright, S. (2006).
\newblock {\em Numerical {O}ptimization}.
\newblock Springer Science \& Business Media.

\bibitem[Pardalos and Schnitger, 1988]{pardalos1988checking}
Pardalos, P.~M. and Schnitger, G. (1988).
\newblock Checking local optimality in constrained quadratic programming is
  {NP}-hard.
\newblock {\em Operations Research Letters}, 7(1):33--35.

\bibitem[Renegar, 1988]{renegar1988polynomial}
Renegar, J. (1988).
\newblock A polynomial-time algorithm, based on newton's method, for linear
  programming.
\newblock {\em {M}athematical {P}rogramming}, 40(1-3):59--93.

\bibitem[Richard~Johnsonbaugh, 2010]{calc}
Richard~Johnsonbaugh, W.~P. (2010).
\newblock {\em Foundations of mathematical analysis}.
\newblock Dover publications, Mineola, New York.

\bibitem[Sorensen, 1982]{sorensen1982newton}
Sorensen, D.~C. (1982).
\newblock Newton's method with a model trust region modification.
\newblock {\em SIAM Journal on Numerical Analysis}, 19(2):409--426.

\bibitem[Sturm, 2002]{sturm2002implementation}
Sturm, J.~F. (2002).
\newblock Implementation of interior point methods for mixed semidefinite and
  second order cone optimization problems.
\newblock {\em Optimization Methods and Software}, 17(6):1105--1154.

\bibitem[Ulbrich, 2004]{ulbrich2004superlinear}
Ulbrich, S. (2004).
\newblock On the superlinear local convergence of a filter-{SQP} method.
\newblock {\em {M}athematical {P}rogramming}, 100(1):217--245.

\bibitem[Vanderbei, 1999]{vanderbei1999loqo}
Vanderbei, R.~J. (1999).
\newblock {LOQO} user's manual---version 3.10.
\newblock {\em Optimization methods and software}, 11(1-4):485--514.

\bibitem[Vicente and Wright, 2002]{vicente2002local}
Vicente, L.~N. and Wright, S.~J. (2002).
\newblock Local convergence of a primal-dual method for degenerate nonlinear
  programming.
\newblock {\em Computational Optimization and Applications}, 22(3):311--328.

\bibitem[W{\"a}chter and Biegler, 2005]{wachter2005line}
W{\"a}chter, A. and Biegler, L.~T. (2005).
\newblock Line search filter methods for nonlinear programming: Motivation and
  global convergence.
\newblock {\em SIAM Journal on Optimization}, 16(1):1--31.

\bibitem[W{\"a}chter and Biegler, 2006]{wachter2006implementation}
W{\"a}chter, A. and Biegler, L.~T. (2006).
\newblock On the implementation of an interior-point filter line-search
  algorithm for large-scale nonlinear programming.
\newblock {\em {M}athematical {P}rogramming}, 106(1):25--57.

\bibitem[Ye, 1991]{ye1991n}
Ye, Y. (1991).
\newblock An ${O}(n^3 {L})$ potential reduction algorithm for linear
  programming.
\newblock {\em {M}athematical {P}rogramming}, 50(1):239--258.

\bibitem[Ye, 1992]{ye1992new}
Ye, Y. (1992).
\newblock A new complexity result on minimization of a quadratic function with
  a sphere constraint.
\newblock In {\em Recent Advances in Global Optimization}, pages 19--31.
  Princeton University Press.

\bibitem[Ye, 1998]{ye1998complexity}
Ye, Y. (1998).
\newblock On the complexity of approximating a {KKT} point of quadratic
  programming.
\newblock {\em {M}athematical {P}rogramming}, 80(2):195--211.

\bibitem[Ye, 2018]{yeTrust}
Ye, Y. (2018).
\newblock {MS\&E}311. {L}ecture note 12.
  https://web.stanford.edu/class/msande311/handout.shtml.

\bibitem[Ye et~al., 1994]{ye1994nl}
Ye, Y., Todd, M.~J., and Mizuno, S. (1994).
\newblock An ${O}(\sqrt{n}{L})$-iteration homogeneous and self-dual linear
  programming algorithm.
\newblock {\em Mathematics of Operations Research}, 19(1):53--67.

\bibitem[Zhang, 1994]{zhang1994convergence}
Zhang, Y. (1994).
\newblock On the convergence of a class of infeasible interior-point methods
  for the horizontal linear complementarity problem.
\newblock {\em SIAM Journal on Optimization}, 4(1):208--227.

\end{thebibliography}
\fi

\end{document}